\documentclass[12pt]{amsart}

\usepackage{amsmath,amsfonts,amssymb,amsthm}

\usepackage[title]{appendix}
\usepackage{pinlabel}
\usepackage{graphicx}
\usepackage[all]{xy}
\usepackage{hyperref}
\usepackage{mathtools}
\usepackage{xcolor}
\xyoption{dvips}

\numberwithin{equation}{section}

\newtheorem{thm}{Theorem}[section]
\newtheorem{cor}[thm]{Corollary}
\newtheorem{prop}[thm]{Proposition}
\newtheorem{lem}[thm]{Lemma}

\theoremstyle{remark}
\newtheorem{rem}[thm]{Remark}
\newtheorem{example}[thm]{Example}
\theoremstyle{definition}
\newtheorem{defn}[thm]{Definition}
\theoremstyle{remark}
\newtheorem{remark}[thm]{Remark}

\newcommand{\R}{\mathbb{R}}

\newcommand{\Z}{\mathbb{Z}}
\newcommand{\C}{\mathbb{C}}

\newcommand{\F}{\mathbb{F}}

\newcommand{\A}{\mathcal{A}}

\begin{document}

\title[Instability of Legendrian knottedness, and non-regular concordances]{Instability of Legendrian knottedness, and non-regular Lagrangian concordances of knots}

\author{Georgios Dimitroglou Rizell}
\author{Roman Golovko}

\begin{abstract}
We show that the family of smoothly non-isotopic Legendrian pretzel knots from \cite{CornwellNgSivek16} that all have the same Legendrian invariants as the standard unknot have front-spuns that are Legendrian isotopic to the front-spun of the unknot. Besides that, we construct the first examples of Lagrangian concordances between Legendrian knots that are not regular, and hence not decomposable. Finally, we show that the relation of Lagrangian concordance between Legendrian knots is  not anti-symmetric, and hence does not define a partial order. The latter two results are based upon a new type of flexibility for Lagrangian concordances with stabilised Legendrian ends.
\end{abstract}

\date{\today}


\maketitle

\section{Introduction and results}
\subsection{Instability of Legendrian knottedness}
\color{black} In this article we investigate a set of Legendrian knots in  $(\R^3,dz-y\,dx)$ that are not smoothly isotopic, admit exact Lagrangian fillings, and have stable-tame isomorphic Chekanov--Eliashberg algebras. (The latter DGAs are not acyclic in view of the fillability.) \color{black} Recall that the standard Legendrian $\tt{tb}=-1$ unknot $U$ in the standard contact $3$-dimensional vector space has zero rotation number and a Chekanov--Eliashberg algebra with $\Z[t^{\pm 1}]$-coefficients generated by a single generator $a$ in degree $|a|=1$, and with differential given by $\partial a=1+t$. \color{black} The first examples of Legendrian knots inside the contact vector-space that are not Legendrian isotopic to $U$, but which have stable-tame isomorphic Chekanov--Eliashberg algebras, were first constructed in \cite{CornwellNgSivek16} by Cornwell--Ng--Sivek. These examples where further explored by Etnyre--Ng in \cite[Section 3.4]{EtnyreNg18}, where it was shown that the Legendrian representatives $\Lambda_m$ of the pretzel knots $P(3,-3,-m)$, $m>3$ described in Figure~\ref{fig:P(3,-3,m)} all satisfy this property. Similarly to $U$, the latter Legendrian knots are also Lagrangian slice. However, they are not doubly-slice as shown by Chantraine--Legout \cite{ChantraineLegout22} \color{black}and, \color{black}  as a consequence, not Lagrangian cobordant to $U$.

\begin{figure}[htb]
\begin{center}
\labellist
\pinlabel $m-3$ at 480 555
\endlabellist
\includegraphics[width=250px]{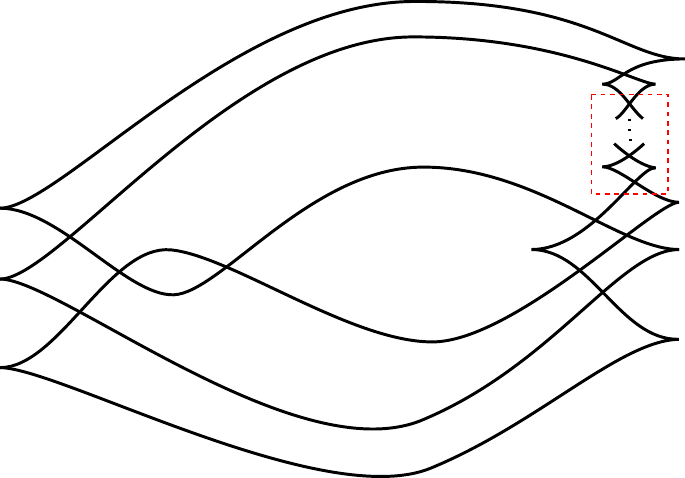}
\caption{The front projection of the Legendrian representative $\Lambda_m$ of the pretzel knot $P(3,-3,-m)$ has $m-3$ crossings in the red box.}
\label{fig:P(3,-3,m)}
\end{center}
\end{figure}

Recall that the front spinning construction (sometimes called  the Legendrian suspension)  is a construction that from a given  Legendrian submanifold $\Lambda$ in the standard contact vector space $\R^{2n+1}$ produces  the Legendrian embedding $\Sigma_{S^k} \Lambda$ of  $\Lambda \times S^k$ in $\R^{2(k+n)+1}$.
This construction for $k=1$ was introduced by Ekholm, Etnyre and Sullivan in \cite{EkholmEtnyreSullivan05}, and then extended to the case of an arbitrary $k$ by the second author in \cite{Golovko14}. This construction admits several \color{black}generalisations in different directions\color{black}, see \cite{BourgeoisSabloffTraynor15, DimitroglouRizellGolovko21, LambertCole14, RutherfordSulivan20}. One can also produce a higher-dimensional Legendrian from a lower-dimensional one by taking the Legendrian product of $\Lambda \subset (Y,\alpha)$ with the exact Lagrangian zero-section $0_M \subset T^*M$. This gives the Legendrian embedding $\Lambda \times 0_M \subset (Y \times T^*M,\alpha+p\,dq)$. When $M=S^n$ and $Y=\R^3$ is the contact vector space, the contact manifold $\R^3 \times T^*S^n$ admits a canonical embedding into $\R^{3+2n}$. The front-spun $\Sigma_{S^n}\Lambda$ can be identified with the image of the Legendrian product under this embedding. Note that both the above Legendrian product, as well as the front-spun, produces a well-defined Legendrian isotopy class, which only depends of the Legendrian isotopy class of the initial Legendrian $\Lambda$.

Here we investigate the Legendrian product operation and front spinning construction when applied to the above family of knots $\Lambda_m$. Since these knots only can be distinguished by smooth topological invariants, while \color{black}the Legendrian invariants \color{black}  agree, it is natural to \color{black} suspect \color{black} that their products and spuns all become Legendrian isotopic; see the following two remarks.

\begin{remark}
When $n \ge 2$, connected $n$-dimensional submanifolds in $(2n+1)$-dimensions that are homotopic are also smoothly isotopic by the classical result of Haefliger \cite{Haefliger}. Hence, the Legendrian products and front spuns of $\Lambda_n$ and $U$ are smoothly isotopic. 
\end{remark}

\begin{remark}
\label{invariantspretzel_unknot}The Legendrian knots $U$ and $\Lambda_m$ all have the same classical invariants except smooth knot class, i.e.~the same rotation number and Thurston--Bennequin invariant. Hence, their spuns have the same classical invariants by \cite[Lemma 4.16] {EkholmEtnyreSullivan05}
and \cite[Lemma 5.1]{Golovko14}.  
Also their Legendrian contact homology DGAs, $\A(\Lambda_m)$ and $\A(U)$, are stable tame isomorphic for all $m \ge 3$ with  arbitrary coefficients. In particular, using the surgery formula from \cite{EkholmLekili}, it follows that the partially wrapped Fukaya categories of the two sectors $\C^2_U$ and $\C^2_{\Lambda_m}$ are equivalent, where these sectors are obtained by putting stops at the Legendrians $U,\Lambda_m \subset \partial_\infty \C^2$. Since the partially wrapped Fukaya category of a product of sectors satisfies a K\"{u}nneth-type formula \cite{GPS}, the partially wrapped Fukaya category of the sector $\C^2_{\Lambda_m} \times T^*S^n$ with stop $\Lambda_m \times 0_{S^n} \subset \partial_\infty (\C^2 \times T^*S^n)$ is thus also equivalent to that of $\C^2_U \times T^*S^n$ with stop $U \times 0_{S^n} \subset \partial_\infty (\C^2 \times T^*S^n)$ whenever $m\ge 3$. Analogous results for the DGA itself should also follow from a version of K\"{u}nneths formula for Legendrian contact homology from work in progress by Strako\v{s} \cite{Strakos}. 
\end{remark}

Our first result is that these Legendrian product and spuns indeed are Legendrian isotopic.
\begin{thm}
\label{spun_iso}
$\Sigma_{S^n} U$ is  Legendrian isotopic to $\Sigma_{S^n} \Lambda_m$ in $(\R^{2(n+1)+1},dz-\mathbf{y}\, d\mathbf{x})$ for $n\geq 1$ and $m>3$. In fact, something stronger is true:  the Legendrian products $U \times 0_{S^n}$ and $\Lambda_m \times 0_{S^n}$ are Legendrian isotopic inside $(\R^3 \times T^*S^n,dz-y\,dx+p\,dq)$ when $n \geq 1$ and $m > 3$. 
\end{thm}

\color{black} In Section \ref{proof_Theorem_spun_iso} \color{black} we give two related, but different constructions of the Legendrian isotopy needed for proving Theorem \ref{spun_iso}. First, note that $\Lambda_m$, $m>3$, can be obtained by an ambient Legendrian surgery $(\Lambda_0)_{\eta_m}$ performed on $\Lambda_0$ consisting of two unlinked Legendrian unknots of $\tt{tb}=-1$ along a Legendrian surgery arc $\eta_m$ with boundary $\partial \eta_m \subset \Lambda_0$; see Figures \ref{fig:handle_decomposition_even} and \ref{fig:handle_decomposition_odd}. The corresponding exact Lagrangian cobordism is described in Figure \ref{fig:filling of pretzel}.

\begin{figure}[t]
\begin{center}
\vspace{3mm}
\labellist
\pinlabel $t$ at 8 700
\pinlabel $\Lambda_m$ at 80 650
\pinlabel ${\color{red} \eta}$ at 230 390
\endlabellist
\includegraphics[width=150px]{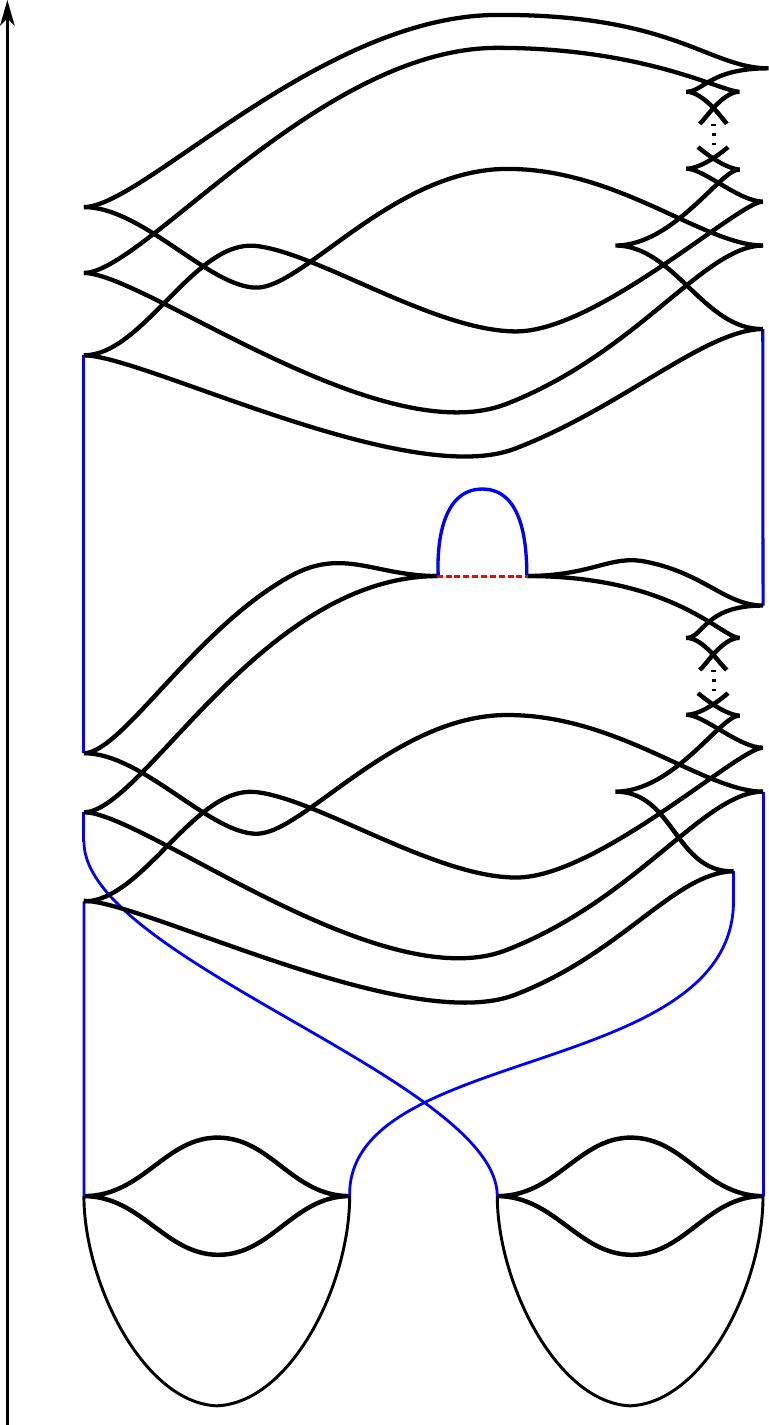}
\caption{A Lagrangian slice disc filling of $\Lambda_m$ that is built from a Lagrangian standard handle-attachment cobordism from two unlinked $\tt{tb}=-1$ unknots to $\Lambda_m$, $m>3$. Here $\eta$ is the Legendrian arc along which the corresponding ambient Legendrian surgery is performed.}
\label{fig:filling of pretzel}
\end{center}
\end{figure}

More precisely, the two crucial properties that we need from this presentation of $\Lambda_m$ are as follows:
\begin{enumerate}
\item The Legendrian arc $\eta_m$  admits a stabilisation in the complement of $\Lambda_0$; and
\item The product of the surgery arc $\eta_m \times 0_{S^n} \subset \R^3 \times T^*S^n$ is formally Legendrian isotopic to the spun $\eta^0_m \times 0_{S^n}$ of a stabilised Legendrian arc $\eta^0_m$ as shown in Figure \ref{fig:unknotodd}, for which the Legendrian ambient surgery $(\Lambda_0)_{\eta^0_m}$ produces the Legendrian standard unknot $U$. Moreover, we need this formal Legendrian isotopy to be supported in the complement of the product $\Lambda_0 \times 0_{S^n}$.
\end{enumerate}

The first strategy uses the fact that the spuns $\eta_m \times 0_{S^n}$ and $\eta^0_m \times 0_{S^n}$ are Legendrian isotopic \color{black} by an isotopy supported in the complement of $\Lambda_0 \times 0_{S^n}$\color{black}; this follows from Murphy's $h$-principle for loose Legendrians \cite{Murphy}, which can be applied since spuns of stabilised Legendrians are loose, see \cite{DimitroglouRizellGolovko14}. 
It then follows by construction of Legendrian ambient surgery and the standard Legendrian neighborhood theorem that $\Lambda_m \times 0_{S^n}$ and $U \times 0_{S^n}$ are Legendrian isotopic.

The second strategy utilises the Weinstein handle-body structure on the complement of the Lagrangian handle-attachment cobordism from the unlinked unknots $\Lambda_0$ to $\Lambda_m$. This is the so-called complementary sectorial cobordism, which is Weinstein since the Lagrangian handle-attachment cobordism is regular in the sense of Eliashberg--Ganatra--Lazarev \cite{EliashbergGanatraLzarev20}; see Section \ref{sec:regular} for the definitions of regularity and complementary cobordism, and Proposition \ref{prop:surgeries} for the regularity. Since the handles of the complementary cobordism are attached along stabilised Legendrian knots, the complementary cobordism becomes flexible after spinning. From this we can produce a symplectomorphism of the symplectisation of $\R^3 \times T^*S^n$ that identifies the standard Lagrangian fillings of $U \times 0_{S^n}$ and $\Lambda_m \times 0_{S^n}$. The Legendrian isotopy is then constructed by applying Lemma \ref{lem:isotopy} to the contactomorphism of the contact boundary that is induced by the symplectomorphism.
\color{black}

\subsection{Non-regular Lagrangian concordances in the symplectisation of $\R^3$}

The existence of a flexible regular \color{black}filling \color{black}is a crucial ingredient for constructing the contact isotopy in Theorem \ref{spun_iso}. While there exist plenty of regular fillings whose Weinstein complements are not flexible, we unfortunately do not have \color{black} a good understanding of when regularity holds for \color{black} Lagrangians inside Weinstein cobordisms. It is expected that all exact Lagrangian fillings in Weinstein domains are regular, as formulated by Eliashberg--Ganatra--Lazarev in \cite[Problem 2.5]{EliashbergGanatraLzarev20}. Furthermore, it was proposed there that any Lagrangian cobordism with a negative end which is not loose or stabilised Legendrian is regular. (In \cite{DimitroglouRizellGolovko19} the authors showed that this condition must be weakened to, at least, the condition that the negative end is not sub-loose, which means that it is not cobordant to a loose Legendrian.)

\color{black} Non-regular Lagrangian cobordisms were first constructed by Eliashberg--Murphy in high dimensions \cite{LagCaps} and Lin \cite{Lin} in dimension four. In these cases, the negative ends are stabilised (which means loose in higher dimensions). These Lagrangian cobordisms are non-regular for the basic topological reason that they have empty positive ends; a regular oriented Lagrangian $L$ inside a Weinstein cobordism $W$ must define a non-zero homology class $[L]\in H_n(\overline{W},\partial_- \overline{W})$ relative the concave boundary of $\overline{W}$.

Here we produce the first examples of Lagrangian concordances of Legendrian knots inside the symplectisation of three-dimensional contact vector space that are not regular; see Section \ref{sec:regular}. Note that the negative Legendrian ends of the constructed Lagrangian concordances are stabilised and, hence, they do not admit Lagrangian fillings.
\begin{thm}[see Theorem~\ref{nonregular_concordance}]
\label{non-reg}
Let $\Lambda$ be a Legendrian knot that admits a decomposable Lagrangian disc filling. Then there exists an exact Lagrangian concordance $C \subset (\R_t \times \R^3,d(e^t\alpha_0))$ with concave Legendrian boundary $\Lambda_-$ obtained by $k$-fold positive and negative stabilisation of $\Lambda$ for $k \gg 0$ sufficiently large, and a convex Legendrian boundary $\Lambda_+$ which is the standard Legendrian unknot $U$ stabilised the same number of times. \end{thm}
The proof of Theorem \ref{non-reg} relies on recent work by the first author \cite{Dimitroglou:Approximations}, where it is  shown that a totally real concordance between Legendrian knots can be approximated by an exact Lagrangian concordance in the same isotopy class, after the knots have been stabilised sufficiently many times  with both positive and negative stabilisations. For investigating the totally real concordances, in Appendix \ref{appendix_totally_real} we prove the analogue of Chantraine's result from \cite{Chantraine12} on the behaviour of classical invariants for Lagrangian concordances in the case of totally real concordances; see Theorem \ref{totallyreal}. 

Showing that the constructed Lagrangian concordances are non-regular is more subtle compared to the situation of cobordisms with empty positive ends, and the obstruction that we use is based upon non-trivial input by Kronheimer--Mrowka in gauge theory; in fact, we need the reformulation \cite[Theorem 3.1]{CornwellNgSivek16} due to Cornwell--Ng--Sivek.
The non-regularity of the concordance follows from our generalisation \color{black}Theorem \ref{thm:regularconc} \color{black}  of \cite[Theorem 3.2]{CornwellNgSivek16}, where it is shown that a smooth concordance from a non-trivial knot to the unknot must have a local maximum, and can in particular not be decomposable in the sense of Chantraine \cite{Chantraine12}. In fact, it also follows that the concordance in Theorem \ref{non-reg} is not ribbon. 

Finally, we show that the non-regular concordances produced by Theorem \ref{non-reg} become Hamiltonian isotopic to the standard cylinder over a loose Legendrian torus after spinning; see Remark \ref{spinconcordance}.

\subsection{Lagrangian concordance of Legendrian knots is not a partial order}

We provide negative answer to the reasonably old question of whether the \color{black} Lagrangian concordance relation between Legendrian isotopy classes of Legendrian knots \color{black} is a partial order \cite{Chantraine10, UpppwerboundLagrcob}. 
\color{black}
By definition, this relation $\Lambda_- \preceq_{conc} \Lambda_+$ holds whenever there exists a Lagrangian concordance with negative and positive end $\Lambda_-$ and $\Lambda_+$, respectively. The relation is clearly reflexive and transitive, and it has been asked whether the relation defines a partial order. The concordances produced in Theorem \ref{non-reg} show that the Lagrangian concordance relation is not anti-symmetric, and hence is not a partial order. More precisely, since any Lagrangian slice disc of a Legendrian $\Lambda_-$ induces a Lagrangian concordance from the standard Legendrian unknot $U$ to $\Lambda_-$, i.e.~$ U \preceq_{conc} \Lambda_-,$ and since the Lagrangian concordance relation can be seen to be preserved under stabilisation of the two knots \cite{Chantraine10}, we get:
\begin{cor}
\label{cor:nonas}
Consider a Legendrian $\Lambda_-$ which is obtained by sufficiently many stabilisations (of both signs) of a Legendrian knot that admits filling by a decomposable Lagrangian slice disc, e.g.~the non-trivial knot $\overline{9_{46}}$. The Legendrian $\Lambda_-$ then admits a concordance both to and from the Legendrian unknot with the same number of stabilisations.
\end{cor}
\section*{Acknowledgments}
The authors are grateful to Baptiste Chantraine, Yakov Eliashberg and Oleg Lazarev for the very helpful discussions.
The first author is supported by the Knut and Alice Wallenberg Foundation under the grants KAW 2021.0191 and KAW 2021.0300, and by the Swedish Research Council under the grant number 2020-04426. The second author is supported by the GA\v{C}R EXPRO Grant 19-28628X and GA\v{C}R Lead Agency Grant 26-20231L.

\section{Preliminaries}

In this section we give the definition of regular Lagrangian cobordisms, which is due to Eliashberg--Ganatra--Lazarev \cite{EliashbergGanatraLzarev20}. Since the complement of a regular Lagrangian cobordism is Weinstein by definition, it can be understood through Kirby calculus, which we also recall.

\subsection{Regular Lagrangian cobordisms}

\label{sec:regular}

A \textbf{compact Liouville cobordism} is a compact exact symplectic cobordism $(\overline{W},d\lambda)$, whose Liouville vector field $\zeta$ is defined uniquely by the equation $d\lambda(\zeta,\cdot)=\lambda(\cdot)$, and which we require to be transverse to the boundary $\partial \overline{W}$; i.e.~the boundary is of contact-type, with contact form given by  \color{black}the restriction \color{black}$\lambda|_{T\partial{\overline{W}}}$. \color{black} By $(\partial_- \overline{W},\alpha_-)$ (resp. $(\partial_+ \overline{W},\alpha_+)$) we denote the concave (resp.~convex) contact boundary, i.e.~the component along which the Liouville vector field points inwards (resp. outwards). Recall that $\partial_+\overline{W} \neq \emptyset$ must hold by Stokes' theorem. The Liouville cobordism is said to be \textbf{Weinstein} if the Liouville vector field is gradient-like for a Morse function.

A \textbf{compact Lagrangian cobordism} in a compact Liouville cobordism is a Lagrangian embedding $(\overline{L},\partial \overline{L}) \subset (\overline{W},\partial \overline{W})$ which is cylindrical near the boundary, i.e.~tangent to the Liouville vector field there. The decomposition $\partial \overline{W}=\partial_+\overline{W} \sqcup \partial_-\overline{W}$ into a convex and concave contact boundary induces a decomposition of the boundary $\partial \overline{L}$ of the compact Lagrangian cobordism into a convex and concave Legendrian boundary
$$\partial _\pm \overline{L}=\partial \overline{L} \cap \partial_\pm \overline{W} \subset (\partial_\pm \overline{W},\alpha_\pm).$$
We will often work with the \textbf{completed} Lagrangian cobordism $L \subset W$ inside the \textbf{completed} Liouville (or Weinstein) cobordism. The completions are constructed by gluing the half-symplectisations
$$ W=((-\infty,0]_t \times \partial_- \overline{W} ,d(e^t\alpha_-)) \:\: \cup \:\:(\overline{W},d\lambda)\:\: \cup \:\: ([0,+\infty)_t \times \partial_+ \overline{W},d(e^t\alpha_+))$$
along the boundary of the compact Weinstein cobordism, and gluing cylindrical ends
$$ L=(-\infty,0] \times \partial_- \overline{L} \:\: \cup \:\:\overline{L}\:\: \cup \:\: [0,+\infty) \times \partial_+ \overline{L}$$
to the compact Lagrangian cobordism.
A completed Lagrangian (resp. Liouville/Weinstein) cobordism will simply be called a Lagrangian (resp. Liouville/Weinstein) cobordism.

Regular Lagrangian cobordisms inside Weinstein cobordisms were introduced by Eliashberg, Ganatra and Lazarev in \cite{EliashbergGanatraLzarev20}. We now recall the definition.
\begin{defn} A Lagrangian cobordism $L \subset W$ from $\Lambda_-$ to $\Lambda_+$ inside a Weinstein cobordism $W$ is  \emph{regular} if the Weinstein structure can be homotoped to one for which the Liouville vector field is tangent to $L$.
\end{defn}
The same authors also proved that if a Lagrangian cobordism is regular, then a Weinstein structure exists for which $L$ satisfies the stronger assumptions given in Definition \ref{defn:regular} below. This stronger type of regularity is also what we mostly will be working with.

Before we give the stronger definition, {\color{black} we recall, following the discussion in \cite[Section 2]{EliashbergGanatraLzarev20}}, that the union
$$\overline{L} \cup \partial_- \overline{W} \subset \overline{W}$$
of a compact Lagrangian cobordism and the concave boundary of a Liouville cobordism admits arbitrarily small compact neighborhoods
$$\overline{W}_L \subset \overline{W} \:\: \text{of} \:\: \overline{L} \cup \partial_- \overline{W} \subset \overline{W}$$
which, when chosen appropriately, \color{black}is a smooth compact \color{black} Weinstein cobordism with concave end $\partial_- \overline{W}_L=\partial_-\overline{W}$, and in which $\overline{L} \subset \overline{W}_L$ is regular. Furthermore, the natural Weinstein structure on $D^*L$ endows $\overline{W}_L$ with a Weinstein structure for which the skeleton of $\overline{W}_L$ is entirely contained inside $L$ and hence, in particular, all critical points of the Liouville vector field are contained in $L$. More precisely, in addition to $\zeta$ being tangent to $L$, we can assume that:
\begin{itemize}
\item all critical points of the Liouville vector field $\zeta$ are contained in $L$; and
\item their descending manifolds are entirely contained in $L$.
\end{itemize}
In other words, the Weinstein handle decomposition of $\overline{W}_L$ is induced from a handle-decomposition of $L$.

We are now ready to recall the definition of a \textbf{special regular Lagrangian} from \cite{EliashbergGanatraLzarev20}.
\begin{defn}
\label{defn:regular}
A Lagrangian cobordism $\overline{L} \subset \overline{W}$ from $\Lambda_- \subset \partial_- \overline{W}$ to $\Lambda_+ \subset \partial_+ \overline{W}$ in a Weinstein cobordism is a \color{black}\textbf{special regular Lagrangian} \color{black} if the Weinstein structure on $\overline{W}$ is homotopic to one obtained from $\overline{W}_L$ by adding Weinstein handles that are attached away from $\Lambda_+ \subset \overline{W}_L$. We call the induced Weinstein cobordism with concave end $\partial_+ \overline{W}_L$ and convex end $\partial_+ \overline{W}$ the \textbf{complementary Weinstein cobordism}, and its handles the \textbf{complementary handles}.
\end{defn}

\begin{rem}
Using the language of Weinstein pairs from \cite{WeinsteinRevisited} and Weinstein sectors from \cite{GanatraPardonShende}, the above can be reformulated in the following way. Consider the sector $W_{\Lambda_+}$ corresponding to the Weinstein pair $(\overline{W},\Lambda_+)$, in which $\tilde{L} \coloneqq L \cap W_{\Lambda_+}$ is naturally included as a compact Lagrangian with boundary $\partial \tilde{L}$ contained inside the sectorial boundary $\partial W_{\Lambda_+}$ of the complete sector. In this setup, $L$ is a special regular Lagrangian if and only if $\tilde{L}$ can be realised as a subset of the skeleton of $W_{\Lambda_+}$ for some homotopic Weinstein structure. In this setting, there is a natural complementary sectorial Weinstein cobordism, which is given as the complement of a neighbourhood of $\tilde{L}$ in the sector.
\end{rem}

Since regular Lagrangians have a complement that is Weinstein, they can be presented and studied by using Legendrian handle-body theory, e.g.~Kirby diagrams \cite{Gompf94} when $\dim W=4$. The regular Lagrangians can also be manipulated via Weinstein homotopy moves which, in this dimension, translate to Kirby moves.

\subsection{Standard moves for Kirby diagrams}
Here we discuss the tools from Kirby calculus that we need to understand and modify the regular cobordisms that we are interested in. We refer to work by Gompf \cite{Gompf94} for a good introduction to Kirby calculus in the setting of contact topology. Our Kirby diagrams will consist of a number of attaching 0-spheres inside $(\R^3_{xyz},dz-ydx)$ which have vanishing $y$--coordinates. The Legendrian arcs are allowed to enter the handles both from the left and right in the front projection, where the arcs are required to have constant $z$--coordinate. The attaching 1-spheres will typically be depicted in red colour, while Legendrian knots will be black. (The latter can be thought of as stops of a Weinstein sector that is built by attaching the one and two-handles.)

In order to relate two different diagrams, we need the following moves that relate different Kirby diagrams for the same Weinstein manifold.

\emph{Move K--0: Legendrian isotopy in the complement of the attaching spheres:} Both Legendrian knots in the contact manifold, as well as attaching 1-spheres, can be deformed by a Legendrian isotopy that supported in the complement of the attaching 0-spheres. By our convention the attaching 0-spheres have a vanishing $y$--coordinate, which means that any Legendrian arc whose front projection has non-zero slope can pass over the small balls where the attaching 0-spheres live. See Figure \ref{fig:m0}.

\begin{figure}[h]
\includegraphics[scale=1.25]{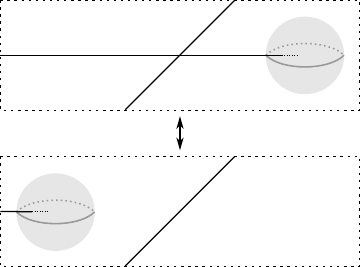}
\caption{The Kirby move K--0: Moving a Legendrian arc of positive slope over a component of the attaching 0-sphere, together with the possible Legendrian arcs entering it.}
\label{fig:m0}
\end{figure}

\emph{Move K--I: Introducing/cancelling a pair of handles:} A Weinstein one- and two-handle that are in cancelling position can be either introduced or removed. This is a standard Kirby move; see \cite[Section 5]{DingGeiges08} in work by Ding and Geiges for the version in contact topology. This move is shown in Figure \ref{fig:m1}.

\begin{figure}[h]
\includegraphics[scale=1.25]{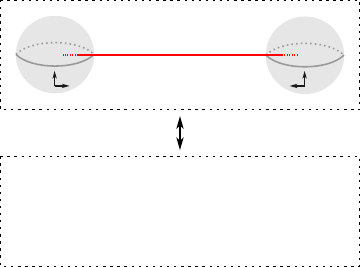}
\caption{The Kirby move K--I: The Kirby move that consists of cancelling a pair consisting of a Weinstein 1-handle and a Weinstein 2-handle in standard position corresponds to the above deformation of the front projection.}
\label{fig:m1}
\end{figure}

\emph{Move K--II: Handle-sliding parallel strands:} Handling sliding parallel strands of a Legendrian over \color{black}a Weinstein \color{black} two-handle which is in cancelling position gives rise to a move shown in Figure \ref{fig:m2}. This is a standard Kirby move; the version in contact manifolds was described by Ding and Geiges in \cite{DingGeiges08}; and Casals and Murphy in \cite{CasalsMurphy}.

\begin{figure}
\includegraphics[scale=1.25]{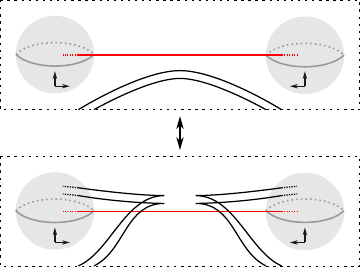}
\caption{The Kirby move K--II: The bottom figure shows the effect in the front projection of the Kirby move that consists of handle-sliding $k=2$ parallel Legendrian strands shown on top over the \color{black}Weinstein \color{black} two-handle whose attaching sphere is the knot shown in red.}
\label{fig:m2}
\end{figure}

\emph{Move K--III: Introducing a double-stabilisation:} Given any number of parallel strands that pass through a single one-handle, one can move the \color{black} bottom strand to the top\color{black}, introducing a double stabilisation of the same. See \cite[Figure 17]{Gompf94} for one version, also its reflection in the z-axis is a possible move. See Figure \ref{fig:m3}.

\begin{figure}
\includegraphics[scale=1.25]{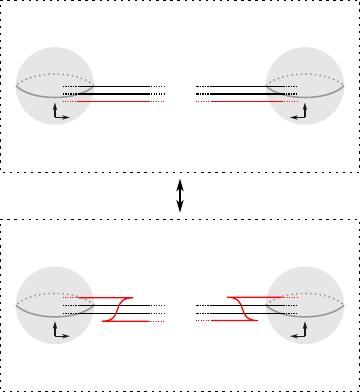}
\caption{The Kirby move K--III: one can cyclically permute the strands going through the handle and introduce a double-stabilization.}
\label{fig:m3}
\end{figure}

\emph{Move K--IV: Moving strands around the attaching sphere:} The bottom strand to the left (resp. right) of the left (resp. right) attaching sphere can be moved to the top of the right (resp. left) side of the same attaching sphere. The strand is given an additional downward cusp-edge in the direction oriented away from the attaching sphere. See Figure \ref{fig:m4} \color{black} for the move in the front projection, and Figure \ref{fig:m4lag} for how the move can be described using the Lagrangian projection.\color{black}

\begin{figure}
\includegraphics[scale=1.25]{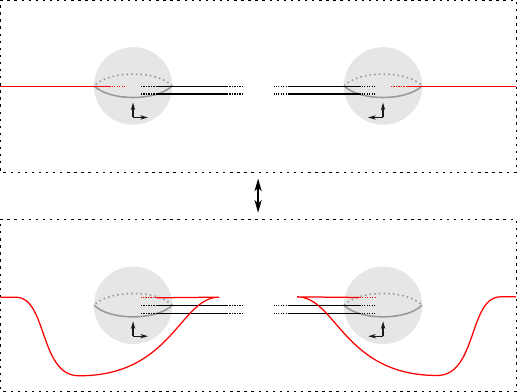}
\caption{The Kirby move K--IV: moving a strand entering the one handle from one side of the attaching sphere to the opposite side.}
\label{fig:m4}
\end{figure}

\begin{figure}
\includegraphics[scale=1.25]{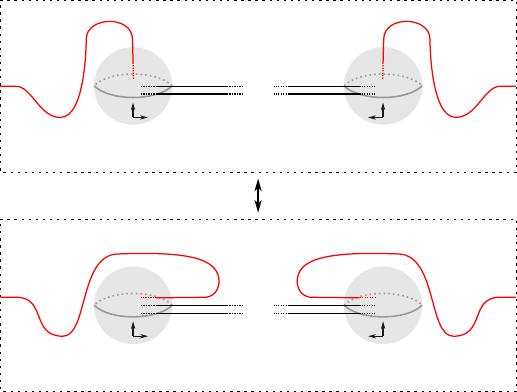}
\caption{The Kirby move K--IV: moving a strand entering the one handle from one side of the attaching sphere to the opposite side shown in the Lagrangian projection. The top shows the rotation of the arcs performed half-way.}
\label{fig:m4lag}
\end{figure}

\section{Regular Lagrangian cobordisms from Legendrian ambient surgery}

The standard technique for constructing a regular Lagrangian cobordisms inside of a symplectisation is by using standard Lagrangian handles that correspond to certain surgery constructions on Legendrians. Notably, in \cite{ConwayEtnyreTosun21}  Conway, Etnyre and Tosun showed that decomposable Lagrangian fillings are regular (see Section \ref{decompreg} below for the definition). In Proposition \ref{prop:surgeries} we generalise this statement to arbitrarily ambient Legendrian surgeries (see Section \ref{sec:ambient} for the definition).

\subsection{Decomposable fillings}
\label{decompreg}

A particularly nice class of Lagrangian cobordisms in $\R \times S^3$ are the decomposable ones, which were defined by Chantraine in \cite{Chantraine12}. By definition, a Lagrangian cobordism is decomposable if it consists of a concatenation of the following simple types of Lagrangian cobordisms:
\begin{itemize}
\item the exact Lagrangian trace-cobordism induced by a Legendrian isotopy;
\item a standard disc-filling of the standard $\tt{tb}=-1$ unknot (this is a zero-handle of the cobordism with an empty negative end), together with other components which all are trivial cylinders; and
\item a Legendrian standard one-handle from a Legendrian knot to the result of a cusp-connected sum performed on it, together with other components which all are trivial cylinders. See Figure \ref{fig:surgery} for the Legendrians before and after a cusp-connected sum, shown at the bottom and top, respectively.
\end{itemize}

\begin{figure}
\labellist
\pinlabel $\color{red}\eta$ at 86 10
\pinlabel $\Lambda_+$ at 48 90
\pinlabel $\Lambda$ at 48 15
\endlabellist
\includegraphics[scale=1.25]{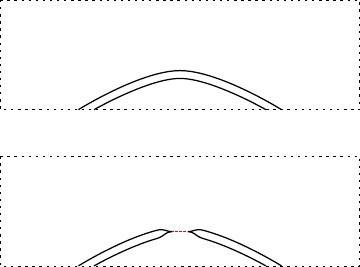}
\caption{Ambient Legendrian surgery:
 The Legendrian $\Lambda$ is deformed in a neighborhood of the surgery arc (\color{black}$\eta$ shown at the bottom\color{black}) to yield the Legendrian \color{black}$\Lambda_+=\Lambda_\eta$ shown on top. \color{black} There is a standard Lagrangian 1-handle cobordism with concave boundary the Legendrian $\Lambda$ before the surgery, and convex boundary the Legendrian $\Lambda_+=\Lambda_\eta$ produced by the surgery.}
\label{fig:surgery}
\end{figure}

\begin{figure}
\labellist
\pinlabel $\Lambda$ at 48 89
\pinlabel $\color{red}\eta$ at 86 86
\pinlabel $\color{red}S_\eta$ at 86 18
\endlabellist
\includegraphics[scale=1.25]{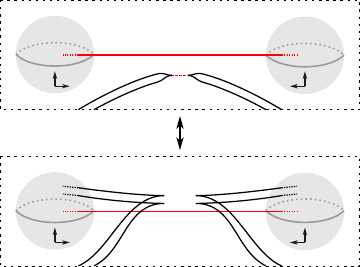}
\caption{Performing the surgery inside a cancelling handle pair: A small surgery-arc $\eta$ that corresponds to a standard cusp-connected sum of two facing cusp edges can be placed inside a one-handle by a K-II move as shown in this figure. We will indicate the fact that the handle contains the surgery-disc $\eta$ by denoting the cancelling one-handle $S_\eta$. The surgery can now be performed inside the handle. cancelling the handle after the surgery by the K-II move again yields the Legendrian $\Lambda_\eta$ resulting by the surgery shown in the bottom of Figure \ref{fig:surgery}.}
\label{fig:surgery2}
\end{figure}

Decomposable Lagrangian cobordisms \color{black} were \color{black} shown to be regular in the work of Conway, Etnyre and Tosun \cite{ConwayEtnyreTosun21}. It is still unknown whether the converse is true or not. The proof in \cite{ConwayEtnyreTosun21} moreover gives an explicit recipe for finding the attaching spheres on the complement of the Legendrian. We proceed to explain this recipe in the case of a decomposable Lagrangian filling.

\emph{The 0-handles:} For each zero-handle we consider $(\C^2,-d^c\rho)$, where $\rho(\mathbf{z})=\frac{\|\mathbf{z}\|^2}{4},$ with the standard filling $\mathfrak{Re}\C^2 \subset \C^2$ of the standard Legendrian unknot $\partial_\infty\left(\mathfrak{Re}\C^2\right) \subset S^3$ of $\tt{tb}=-1$. This filling is clearly is tangent to the radial Liouville vector field. In the general case, when the Lagrangian cobordism has several 0-handles, one has to consider a Weinstein connected sum of such Weinstein four-balls. Since this Weinstein connected sum again produces the standard sphere, and since these handles can be determined from the number of zero-handles, we will omit these 1-handles from the diagram.

\emph{The 1-handles:} We add a pair of cancelling Weinstein one and two-handles just above the region where the cusp-connected sum takes place; this is move K--I shown in Figure \ref{fig:m1}. The resulting diagram after having performed the cusp-connected sum is shown at the top of Figure \ref{fig:m2}. \color{black}Performing the K--II move as shown in the same figure places the two strands of the Legendrian that underwent cusp-connected sum inside the Weinstein 1-handle. (Alternatively, one could handle-slide the two facing cusp-edges together with the surgery arc into the one-handle similarly to the K--II move as shown in Figure \ref{fig:surgery2}.) \color{black} It can readily be seen that the corresponding standard Lagrangian one-handle in this position \color{black} can be made \color{black} tangent to the standard Liouville structure on the Weinstein one-handle. In particular, the critical point of the Weinstein one-handle coincides with the critical point of the standard Lagrangian one-handle. See Proposition \ref{prop:surgeries} for the general statement in arbitrary dimensions.

\subsection{The case of a general ambient Legendrian surgery}
\label{sec:ambient}

In \cite{Dimitroglou:Ambient} the first author generalised the construction of cusp-connected sum to the construction of a Legendrian ambient $k$-surgery on a Legendrian $\Lambda^n \subset Y^{2n+1}$, for $k \le n$. Legendrian ambient $k$-surgery produces an embedding $\Lambda_\eta \subset Y$ of the result of a $k$-surgery on $\Lambda$ along a framed sphere $\partial \eta \subset \Lambda$, where $\eta\subset Y$ is a choice of a so-called isotropic surgery $k+1$-disc. The choice of framing of the embedded sphere $\partial \eta$, i.e.~a trivialisation of its normal bundle inside $\Lambda$, is also a part of the data, together with a trivialisation of the conformal symplectic normal bundle of $\eta$. The isotropic surgery disc $\eta$ is required to intersect the Legendrian $\Lambda$ precisely along its boundary, where it satisfies the property that the direction in $\eta$ that is normal to $\partial \eta$ also is normal to the Legendrian $\Lambda$. In addition, the trivialisation of the conformal symplectic normal bundle of $\eta$ and the framing of $\partial \eta \subset \Lambda$ are assumed to be compatible in the following sense: the trivialisation of the conformal symplectic normal bundle of $\eta$ is required to be constant along $\partial \eta$ with respect to its framing.

The upshot with the more general definition of Legendrian ambient surgery compared to the cusp-connected sum is that the former construction does not require the Legendrians to be in some particular position; up to contactomorphism the surgery only depends on the choice of Legendrian $\Lambda$, the isotropic surgery disc $\eta$, and the choice of framings of $\partial \eta \subset \Lambda$ and $\eta \subset Y$ up to homotopy.

The Legendrian embedding $\Lambda_\eta$ produced by the ambient Legendrian surgery can be assumed to live inside an arbitrarily small neighbourhood of $\Lambda \cup \eta$. In addition, the construction produces an exact Lagrangian handle-attachment cobordism $L_\eta$ from $\Lambda$ to $\Lambda_\eta$ on which the restriction $t|_{L_\eta}$ of the symplectisation coordinate $t$ on $\R_t \times Y$ has a unique Morse critical point of index $k+1$.

We refer to \cite[Section 4]{Dimitroglou:Ambient} for the precise description of the construction. The most important feature of the general construction that we rely on here is that the construction is performed by implementing a standard model for the Lagrangian handle in the symplectisation of a standard contact neighbourhood of the isotropic disc $\eta$.

In the case $n=k+1=1$ the construction of the ambient Legendrian surgery in a neighbourhood of $\eta$ is depicted in Figure \ref{fig:surgery3}; in this case there is no choice of conformal symplectic normal bundle.

In general, the standard neighbourhood is constructed by using the choice of trivialisation of the conformal symplectic normal bundle of $\eta$, and is of the form
$$([-\epsilon,\epsilon]_z \times D_\epsilon^*\eta \times D_\epsilon^*D^{n-k-1},dz-\lambda_\eta-\lambda_{D^{n-k-1}}) \hookrightarrow (Y^{2n+1},\alpha),$$
where by $\lambda_M \in \Omega^1(T^*M)$ we denote the tautological one-form. More precisely, the surgery $k+1$-disc is identified with the subset
$$\{0\} \times 0_\eta \times \{0\} \subset \{0\} \times  0_\eta \times 0_{D^{n-k-1}} \subset [-\epsilon,\epsilon]_z \times D^*\eta \times D^*D^{n-k-1}$$
of the zero-section in the above standard neighbourhood, while the Legendrian $\Lambda$ is identified with the Lagrangian co-normal bundle $N^*_{(\partial \eta) \times D^{n-k-1}}=N^*_{\partial \eta} \times 0_{D^{n-k-1}}$ of $(\partial \eta) \times D^{n-k-1}$, lifted to have vanishing $[-\epsilon,\epsilon]_z$-coordinate.

The choice of conformal symplectic normal bundle induces an identification of the $D^*D^{n-k-1}$-factor in the above neighbourhood. Since $N^*_{(\partial \eta)\times D^{n-k-1}}$ is a trivial $\R$-fibration over $(\partial \eta)\times  D^{n-k-1}$, the compatibility of the framing of $\partial \eta$ and the above conformal symplectic normal bundle implies that the $\R$-fibre of the conormal bundle together with the standard coordinates on the $0_{D^{n-k-1}}$-factor together yield the framing of $\partial \eta \subset \Lambda$.

\begin{figure}
\labellist
\pinlabel $\color{red}\eta$ at 88 153
\pinlabel $\Lambda$ at 8 160
\pinlabel $\Lambda'$ at 8 100
\pinlabel $\Lambda_\eta$ at 8 40
\endlabellist
\includegraphics[scale=1.25]{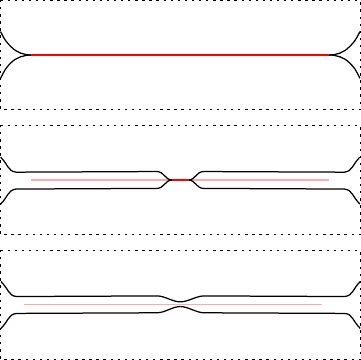}
\caption{The top depicts the front projection of the Legendrian link $\Lambda$ and the surgery arc $\eta$ in the standard jet-neighborhood of a surgery-arc. The middle shows the Legendrian $\Lambda'$ obtained by a Legendrian isotopy of $\Lambda$ supported near $\eta$ that moves the two cusp-edges closer to each other. The bottom picture depicts the result of a cusp-connected sum of $\Lambda'$. This process is the definition of an ambient Legendrian surgery $\Lambda_\eta$ of $\Lambda$ defined by the surgery arc $\eta$.}
\label{fig:surgery3}
\end{figure}

Passing to the symplectisation, we can find a neighbourhood of $[0,1] \times \eta \subset \R \times Y$ in the symplectisation of the form
$$ D^*[0,1] \times D^*\eta \times D^*D^{n-k-1} $$
in which the cone $[0,1] \times \eta$ is identified with $0_{[0,1]} \times 0_\eta \times \{0\}$ and the cone $[0,1] \times \Lambda$ is identified with $0_{[0,1]} \times N^*_{\partial \eta} \times 0_{D^{n-k-1}}$. The standard Lagrangian $k+1$-handle is constructed to be a cylinder over $\Lambda$ outside the above neighbourhood of the symplectisation, while the intersection with the above neighbourhood is an explicitly constructed neighbourhood of the core-disc; see \cite[Section 4.2.2]{Dimitroglou:Ambient}. We will only need a description of $L_\eta$ in a neighbourhood of the zero-section
$$0_{[0,1]} \times 0_\eta \times 0_{D^{n-k-1}} \subset D^*[0,1] \times D^*\eta \times D^*D^{n-k-1}.$$
Namely, in a small neighbourhood of this zero-section, the standard handle $L_\eta$ is given by the Lagrangian co-normal bundle of a $k+1$-dimensional disc
$$\overline{C}_\eta=\Gamma \times \{0\} \subset [0,1] \times \eta \times D^{n-k-1}$$
where $\Gamma \subset [0,1]\times\eta$ is a smooth codimension one embedding of a disc given as the union of $[0,1/2] \times \partial \eta$ and a suitable graph of a function $\eta \to [1/2,1]$. In particular, this means that $L_\eta$ intersects the zero-section of this standard neighbourhood in a representative $\overline{C}_\eta \subset \overline{L}_\eta$ of the $k+1$-dimensional core-disc of the handle. Further, we note that the core disc $\overline{C}_\eta$ is the boundary of an isotropic $k+2$-disc $\overline{D}_\eta \subset 0_{[0,1]} \times 0_{\eta} \times \{0\}$ given by the sub-level set of the above graph. Here $\overline{D}_\eta$ is homeomorphic to a disc and has boundary with corners; one boundary stratum is the disc $\overline{C}_\eta$, while the other stratum is $\{0\} \times 0_\eta \times \{0\} \subset 0_{[0,1]}\times 0_\eta \times \{0\}$. By
$$C_\eta=\overline{C}_\eta \cup \left((-\infty,0] \times \partial \eta \right)\:\: \text{and} \:\: D_\eta=\overline{D}_\eta \cup \left((-\infty,0] \times \partial \eta\right)$$
we denote the cylindrical extensions of these discs. In the coordinates used in \cite[Section 4.2.2]{Dimitroglou:Ambient} the $k+1$-disc $C_\eta$ is given by the cusp edge of the front projection intersected with $\{x_{k+2}=\ldots=x_n=0\}$ while the $k+2$-disc $D_\eta$ is the subset of $T^*\R^{n+1}$ that is contained in the zero-section intersected with $\{x_{k+2}=\ldots=x_n=0\}$, with boundary $C_\eta$. 

The following lemma is a direct consequence of the construction in terms of a local model implemented in a choice of standard neighbourhood.
\color{black}
\begin{lem}[\cite{Dimitroglou:Ambient}]
 \color{black} The Legendrian isotopy class of the resulting Legendrian $\Lambda_\eta$ only depends on $\Lambda\cup\eta$, together with the data of the framings, up to contact isotopy. The Hamiltonian isotopy class of the resulting Lagrangian standard handle-attachment cobordism
$$(\overline{L}_\eta,\partial \overline{L}_\eta) \subset ([0,1] \times Y,(\{0\} \times \Lambda )\cup (\{1\} \times \Lambda_\eta))$$
also satisfies the same dependence.
\end{lem}

\color{black}
In the special case when $\dim\Lambda=\dim\eta=1$, the Legendrian $\Lambda_\eta$ is constructed by performing a cusp-connected sum in a standard neighbourhood $J^1\R$ of the arc $\eta$ as outlined in Section \ref{sec:cuspconn} below.

Using the local model for the surgery we prove the following generalisation of the  result from \cite{ConwayEtnyreTosun21} by Conway, Etnyre and Tosun, where they showed that a decomposable Lagrangian cobordism is regular.
\begin{prop}
\label{prop:surgeries}
Let $\eta \subset Y^{2n+1}$ be an isotropic surgery disc of dimension $k+1 \le n$ with boundary $\partial \eta \subset \Lambda$. The corresponding standard Lagrangian $k+1$-handle cobordism $L_\eta \subset \R \times Y^{2n+1}$ inside the symplectisation is regular; this is a Lagrangian cobordism with concave Legendrian boundary $\Lambda$, and convex boundary the Legendrian $\Lambda_\eta$ produced by the surgery. More precisely, the symplectisation admits a Weinstein structure for which:
\begin{enumerate}
\item there are two Weinstein handles in cancelling position of index $k+1$ and $k+2$;
\item $L_\eta$ is contained in the Weinstein $k+1$-handle and is tangent to the Liouville flow, where the critical point of the Liouville flow of the handle coincides with the critical point of the restriction $t|_{L_\eta}$ of the symplectisation coordinate;
\item The descending manifold of the Weinstein $k+1$-handle coincides with the core-disc $C_\eta$ consisting of $\overline{C}_\eta \subset \overline{L}_\eta$ described above with the cylinder $(-\infty,0] \times \partial \eta$ adjoined;
\item The closure of the descending manifold of the Weinstein $k+2$-handle is the isotropic $k+2$-disc $D_\eta$ consisting of $\overline{D}_\eta$ described above with the cylinder $(-\infty,0] \times \eta$ adjoined. 
\end{enumerate}
In the special case when $n=1$ and $k=0$, the corresponding Kirby diagram describing the Legendrian knot $\Lambda_\eta$ and the handles can be obtained by replacing the top of Figure \ref{fig:darbouxball} (which describes $\Lambda \cup \eta$ in a neighbourhood of $\partial \eta$) with the diagram at the bottom of the same figure.
\end{prop}
\begin{proof}
Recall that the standard Weinstein index-$i$ handle of dimension $2(n+1)$ is given by the standard symplectic structure on
$$ (D^*D^i \times D^{2(n+1-i)},d(p\,dq)+\omega_0)$$
endowed with a Liouville form for which $\{(0,0)\}$ is a unique critical point of index $i$, whose descending manifold (i.e.~the skeleton of the handle, also known as core-disc) is the isotropic $i$-disc $0_{D^i} \times \{0\}$, while its ascending manifold (i.e.~co-core disc) is the coisotropic $2n-i$-disc $D^*_0 D^i \times D^{2(n+1-i)}$. Moreover, the Liouville vector-field can be taken to be the radial Liouville vector field corresponding to the Liouville form $-d^c\|\mathbf{z}\|^2/4$ in the second component.

We begin by symplectically identifying a neighbourhood of the isotropic $k+1$-disc $C_\eta$ with a standard Weinstein $k+1$-handle embedded in $\R \times Y$, which has been attached along the isotropic sphere $\partial \eta \subset \partial((-\infty,0] \times Y)$. Here the framing of $\partial \eta$ in $Y$ is induced from the framing of $\partial \eta \subset \Lambda$. Then we denote by $\overline{X}_a \subset \R \times Y$ the embedded Weinstein sub-cobordism consisting of $(-\infty,0] \times Y$ together with the attached Weinstein $k+1$-handle. This sub-cobordism has a smooth contact boundary $\partial \overline{X}_a \subset \R \times Y$ along which the Liouville vector field induced by the handle-attachment is outwards pointing. (Note that this is different from the original Liouville vector field $\partial_t$ on the symplectisation.)

\begin{figure}[t!]
\begin{center}
\labellist
\pinlabel $\color{red}\eta$ at 125 160
\pinlabel $\color{red}S_\eta$ at 125 55
\pinlabel $\Lambda$ at 5 165
\pinlabel $\Lambda_\eta$ at 6 60
\endlabellist
\includegraphics[scale=1.25]{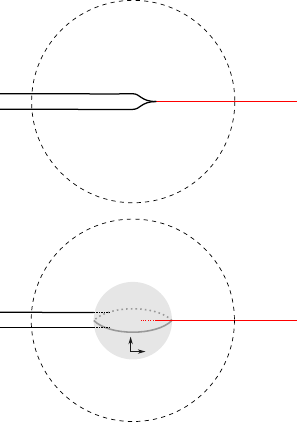}
\caption{Top: The front projection in a Darboux ball that depicts a surgery-arc $\eta$ with boundary on a Legendrian $\Lambda$. Bottom: A Kirby diagram where a one-handle has been attached along $\partial \eta$, a cancelling two-handle has been attached along $S_\eta$ which coincides with $\eta$ outside the attaching region, and where $\Lambda_\eta$ can be obtained by moving the cusp edges of $\Lambda$ containing $\partial \eta$ towards each other in the one-handle, inside of which a cusp-connected sum is performed. Note that $\Lambda_\eta$ and $S_\eta$ enter the handles on opposite sides of the attaching sphere.}
\label{fig:darbouxball-before}
\end{center}
\end{figure}

\begin{figure}[t!]
\label{comp_surg_arc}
\begin{center}
\labellist
\pinlabel $\color{red}\eta$ at 125 160
\pinlabel $\color{red}S_\eta$ at 125 55
\pinlabel $\Lambda$ at 5 165
\pinlabel $\tilde\Lambda_\eta$ at 6 60
\endlabellist
\includegraphics[scale=1.25]{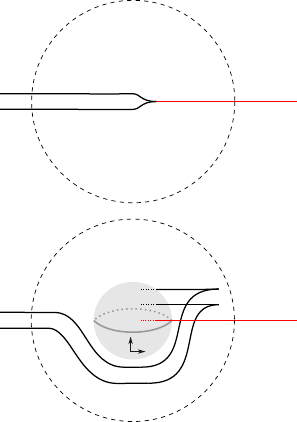}
\caption{The Kirby diagram on the bottom is obtained from the one on the bottom of Figure \ref{fig:darbouxball-before} by a contact isotopy which is the identity in a neighbourhood of $S_\eta$, is supported in the one-handle, and which moves $\Lambda_\eta$ so that it enters the one-handle on the same side as $S_\eta$, with the two strands having larger $z$--coordinates.} 
\label{fig:darbouxball}
\end{center}
\end{figure}

Further, we may assume that this identification takes $L_\eta \cap \overline{X}_a$ to the Lagrangian
$$L_0 \coloneqq 0_{D^{k+1}} \times (\mathfrak{Im}\C^{n-k} \cap D^{2(n-k)}) \subset (D^*D^{k+1} \times D^{2(n-k)},d(p\,dq)+\omega_0) $$
contained in the standard Weinstein $k+1$-handle that is everywhere tangent to the Liouville vector field, while $D_\eta \cap \overline{X}_a$ inside the handle is identified with
$$0_{D^{k+1}}  \times \left(\{x_1\ge 0,x_2=\ldots=x_{n-k}=y_1=\ldots=y_{n-k}=0\} \cap D^{2(n-k)}\right).$$
The intersection of $D_\eta$ with the contact boundary $\partial \overline{X}_a \subset \R \times Y$ can, moreover, be assumed to be an isotropic $k+1$-sphere. 

We may symplectically identify a neighbourhood of $D_\eta \setminus X_a$ with a Weinstein $k+2$-handle, under which $D_\eta \setminus C_\eta$ is the isotropic core-disc of the handle. Using $\overline{X}_b \supset \overline{X}_a$ to denote the union of $\overline{X}_a$ together with a neighbourhood of $D_\eta$, we have produced a symplectic identification of the neighbourhood $\overline{X}_b \subset \R \times Y$ with the result of attaching Weinstein handles of indices $k+1$ and $k+2$ on $\partial (-\infty,0]\times Y$ in cancelling position. In particular, $\partial \overline{X}_b$ is contactomorphic to $Y$.

Finally, after a compactly supported deformation of the Liouville vector field $\partial_t$ of $(\R \times Y,d(e^t\alpha))$ to a complete Liouville vector field $\zeta$ without critical points for the same symplectic form, we may assume that $L_\eta \setminus X_b$ is tangent to $\zeta$ in $\R \times Y \setminus \overline{X}_b$, and that $\zeta$ coincides with the Liouville vector field induced by the Weinstein handle-decomposition of $\overline{X}_b$ near its boundary $\partial \overline{X}_b$. This gives us the sought Weinstein structure on $\R \times Y$ and thus finishes the proof of properties (1)--(4).

In the special case when $n=1$ and $k=0$, the handles in cancelling position together with the Legendrian $\Lambda_\eta$ consists of:
\begin{itemize}
\item A Weinstein one-handle attached at the 0-sphere $\partial \eta$;
\item A Weinstein two-handle attached along a knot $S_\eta$ that coincides with $\eta$ outside of the above $0$-handle; while
\item The Legendrian $\Lambda_\eta$ coincides with $\Lambda$ outside of the one-handle and coincides with a cusp-connected sum of two strands inside the one-handle.
\end{itemize}

In Darboux balls that contain $\partial \eta$ the Legendrian $\Lambda_\eta$ and attaching $0$-spheres and $1$-spheres $S_\eta$ can be described as shown on the bottom in Figure \ref{fig:darbouxball-before}.

Note that this Kirby diagram is not in standard form, since there are strands entering the one-handle from both the left and right. This we correct by the contact isotopy described in K--IV move applied to $\Lambda_\eta$; see Figure \ref{fig:m4}. This contact isotopy is supported inside the one-handle, fixes the attaching sphere $S_\eta$ of the cancelling two-handle, and places the strands that correspond to $\Lambda_\eta$ on the same side as the strand corresponding to $S_\eta$. This produces the configuration shown on the bottom of Figure \ref{fig:darbouxball}, as sought.
\end{proof}

\subsection{Ambient surgery, cusp-connected sum, and isotopies.}
\label{sec:cuspconn}

Since the main focus here is Legendrian ambient $0$-surgery on Legendrian knots in $\R^3$ we will give some additional details in this case.

When $\eta$ connects two cusp edges in the front projection of the Legendrian link $\Lambda$, the Legendrian ambient surgery can be performed in the following manner.

\emph{Step (1):} Choose a standard contact jet-neighbourhood $J^1\eta$ of the surgery arc in which the Legendrian $\Lambda$ consists of two facing cusp-edges, whose singularities are connected by the zero-section $\eta \subset J^1\eta$ that corresponds to the arc $\eta$.

\emph{Step (2):} Perform a Legendrian isotopy of $\Lambda$ whose support is compact inside the above neighbourhood $J^1\eta$, so that the cusp-edges become arbitrarily close to each other; denote the resulting Legendrian by $\Lambda'$. We perform this isotopy in an arbitrarily small neighborhood of $\eta$, while fixing the arc $\eta$ set-wise, and so that it induces an isotopy of front-diagrams; See the top and middle front in Figure \ref{fig:surgery3}.

\emph{Step (3):} Perform a standard cusp-connected sum at the two facing cusp-edges of $\Lambda'$ that now are joined by a very small segment of $\eta$, and can be assumed to be arbitrarily close. The result is shown in the bottom of Figure \ref{fig:surgery3}.

Moreover, the ambient Legendrian surgery gives rise to a \emph{Lagrangian standard handle-attachment cobordism} that consists of the trace of the aforementioned isotopy, followed by a Lagrangian standard one-handle. 

The isotopy performed in Step (2) is easy to describe also in the front of $\R^3$; we express this in a lemma. \color{black} The techniques, and moves shown in the figures, are covered by the generalised Reidemeister moves for Legendrian graphs developed in \cite{ODonnolPavelescu2012}, which also have been used by Sabloff, Vela-Vick, Wong, and Wu in \cite{SabloffVelaVickWongWu}.
\color{black}
\begin{lem}
\label{step2front}
For a generic choice of surgery arc $\eta$ and Legendrian $\Lambda$, the result $\Lambda_\eta$ after Legendrian ambient surgery can be obtained by first applying the Legendrian isotopy of $\Lambda$ in Step (2) by performing consecutive moves of type (S1) and (S2) shown in Figure \ref{fig:surgarc} to yield $\Lambda'$, and then performing a standard cusp-connected sum on the resulting $\Lambda'$.
\end{lem}
\begin{figure}
\labellist
\pinlabel $\color{red}\eta$ at 100 122
\pinlabel $\color{red}\eta$ at 100 36
\pinlabel $\color{red}\eta$ at 140 94
\pinlabel $\color{red}\eta$ at 140 8
\pinlabel $(S1)$ at 44 71
\pinlabel $(S2)$ at 171 71
\endlabellist
\includegraphics[scale=1.25]{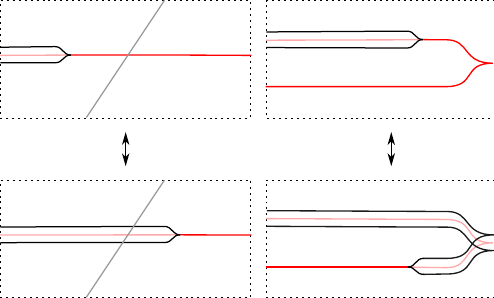}
\caption{The Legendrian isotopy of $\Lambda$ following the surgery arc. \color{black} These moves were also established by Sabloff, Vela-Vick, Wong and Wu, see e.g.~\cite[Figure 9]{SabloffVelaVickWongWu} for (S2).}
\label{fig:surgarc}
\end{figure}

\section{Spinning surgery discs and attaching spheres (Proof of Theorem \ref{spun_iso})}
\label{proof_Theorem_spun_iso}

In this section we provide two different, but related, strategies that produce Legendrian isotopies after ``spinning'' a Legendrian $\Lambda \subset Y^{2m+1}$ or, more generally, after passing to the product Legendrian
$$ \Lambda \times 0_{S^n} \subset Y^{2m+1} \times T^*S^n.$$
These strategies will  lead to the proof of Theorem \ref{spun_iso}.
\color{black}
We will mainly be interested in the case then $Y^{2m+1}=\R^3$. We begin by showing some results that can be of general use.

\subsection{Generalities on formal Legendrian isotopies in the complement of a Legendrian}

\begin{prop}
\label{charact_rel}
\color{black}Let $\Lambda\subset Y^3$ denote a fixed Legendrian link. 
Further, let $S_0,S_1 \subset Y^3 \setminus \Lambda$ be two Legendrians in the complement of $\Lambda$, where each component of $S_i$ is either closed or compact with boundary. If there exists a Legendrian isotopy $S_t \subset Y$ which takes $S_0$ to $S_1$ (the isotopy is allowed to cross $\Lambda$), which is the identity near the boundary $\partial S_0 \cup \partial S_1$, then
\color{black}
it follows that:
\begin{itemize}
\item When $n=2l+1>0$ is odd, the two products $S_i \times 0_{S^n}$, $i=0,1$, are (compactly supported) formally Legendrian isotopic when considered inside the contact manifold $(Y\times T^*S^n) \setminus (\Lambda \times 0_{S^n})$.
\item When $n=2l>0$ is even, the same is true under the additional assumption that any two components of the immersed trace cobordism
$$\{ (t,y); \:\: y \in S_t \cup \Lambda \} \subset [0,1] \times Y$$
have vanishing algebraic intersection number.
\end{itemize}
\end{prop}
\begin{proof}
In the case when $n$ is odd, we can find an explicitly defined perturbation that removes all intersections of $S_t\times 0_{S^n}$ and $\Lambda \times 0_{S^n}$. More precisely, all these intersections are clean with intersection locus of the form $\{\mathrm{pt}\} \times 0_{S^n}$. By the addition of a small non-vanishing section of $T^*S^n \to S^n$ along the intersection loci $\{\mathrm{pt}\} \times 0_{S^n} \subset S_t \times 0_{S^n}$, we obtain a formal Legendrian isotopy from $S_0 \times 0_{S^n}$ to $S_1 \times 0_{S^n}$ that lives inside the complement of $\Lambda \times 0_{S^n}$.

In the case when $n$ is even, we argue as follows. The Legendrian isotopy $S_t \times 0_{S^n}$ intersects $\Lambda \times 0_{S^n}$ cleanly in submanifolds of the form $\{\mathrm{pt}\} \times 0_{S^n}$. After a generic perturbation of the isotopy, we obtain a situation where the corresponding intersections of the immersed trace cobordism
$$\mathcal{T} \coloneqq \{ (t,y); \:\: y \in S_t \cup \Lambda \}  \times 0_{S^n} \subset \R \times Y \times T^*S^n$$
are generic double-points where, moreover, the algebraic intersection number of any two fixed components vanishes.

We then deform the isotopy $S_t \times 0_{S^n}$ through isotopies to one that avoids $\Lambda \times 0_S^n$ by the following argument using the existence of certain Whitney discs:

Create Whitney disc by connecting two double points of opposite signs with arcs in both intersecting sheets of the immersed trace cobordism $\mathcal{T}$; this forms a loop in $[0,1]_t\times Y \times T^*S^n$. We may assume that each arc intersects each constant $t$-slice $\{t\} \times Y \times T^*S^n$ in at most one point. Since the Legendrians have codimension strictly greater than two in the contact manifold $Y \times T^*S^n$, one can find a smooth family of paths in each slice $\{t\} \times Y \times T^*S^n$ that connect the two unique points on the arcs in the sheets of $\mathcal{T}$. The high codimension makes it possible to assume that these family of paths trace out an embedded two-dimensional disc with boundary on $\mathcal{T}$ and interior disjoint from $\mathcal{T}$. These Whitney discs can be used to deform $S_t \times 0_{S^n}$ in the required manner, to remove any intersection with $\Lambda \times 0_{S^n}$.

Finally, since the initial isotopy was a Legendrian isotopy, the bundle maps can be extended over the family of deformed isotopies that we just constructed, in order to yield the sought formal Legendrian isotopy.
\end{proof}

Using the fact that products of stabilised knots become loose \cite[Proposition 1.13]{DimitroglouRizellGolovko14}, combined with Murphy's $h$-principle for loose Legendrian embeddings \cite{Murphy}, we get:

\begin{cor}
\label{cor:legiso1}
If the two Legendrians $S_0,S_1 \subset Y^3 \setminus \Lambda$ satisfy the assumptions of Proposition \ref{charact_rel} and, moreover, are stabilised when considered inside the contact three-manifold $Y^3 \setminus \Lambda$, then the two Legendrians $S_i \times 0_{S^n}$, $i=0,1,$ are Legendrian isotopic when considered inside the contact manifold  $(Y\times T^*S^n) \setminus (\Lambda \times 0_{S^n})$.
\end{cor}

\begin{cor}
\label{cor:legiso2}
If the two Legendrians $S_0,S_1 \subset Y^3 \setminus \Lambda$ satisfy the assumptions of Proposition \ref{charact_rel} and, moreover, are stabilised when considered inside the contact three-manifold $Y^3 \setminus \Lambda$, then the two Legendrians $\Lambda \times 0_{S^n} \subset Y_{S_i} \times T^*S^n$, $i=0,1$, are \color{black} compactly supported \color{black} contactomorphic. When $Y_{S_i} \cong \R^3$, then the induced Legendrians inside $\R^3 \times T^*S^n$ are Legendrian isotopic.\end{cor}

\color{black} The existence of a Legendrian isotopy in $\R^3
\times T^*S^n$ is a direct consequence of the existence of a compactly supported contact contactomorphism. Namely, since any compact subset in $\R^3$ can be displaced from itself by a contactomorphism, any compactly supported contactomorphism of $\R^3 \times T^*M$ preserves the Legendrian isotopy class of any compact Legendrian; see Lemma \ref{lem:isotopy} below.\color{black}

We are now ready to prove Theorem~\ref{spun_iso}, which follows from either one of the following two constructions.

\subsection{First strategy: Spinning the ambient Legendrian surgery discs}
\label{first_str}

We begin with the following basic result.

\begin{lem}
\label{lem:isosurgery}
Consider a Legendrian $\Lambda \subset Y^3$ together with two surgery arcs $\eta_i \subset Y$ that have boundary on $\Lambda$ and agree near the boundary. If the two Legendrians
$$\eta_i \times 0_{S^n} \subset Y^3 \times T^*S^n, \:\:i=0,1,$$
are isotopic by a contact isotopy that fixes a neighbourhood of $\Lambda \times 0_{S^n}$, then the two spuns $\Lambda_{\eta_i} \times 0_{S^n}$, $i=0,1,$ are Legendrian isotopic.
\end{lem}
\begin{proof}
The ambient Legendrian surgery is constructed in a small standard neighbourhood $J^1\eta_i$ of $\eta_i$ by implementing the Legendrian shown in Figure \ref{fig:surgery3}. Since standard neighbourhoods are unique up to contact isotopy, the construction does not depend on the choice of neighbourhood.

Since there is a canonical identification
$$(J^1\eta_i \times T^*S^n,dz+p\,dq+\lambda_{S^n})=J^1(\eta_i \times S^n),dz+p\,dq+\lambda_{S^n})$$
of contact manifolds, a product of a contact standard neighbourhood of $\eta_i$ with a Weinstein neighbourhood of $0_S^n$ is a standard neighbourhood of the Legendrian $\eta_i \times 0_{S^n}$. Any two contact neighbourhoods of $\eta_i \times 0_{S^n}$, $i=0,1,$ are contact isotopic; here we use the uniqueness of contact neighbourhood together with the assumption that these two Legendrians are isotopic.

It follows that the product Legendrians are isotopic as sought.
\end{proof}

Now consider the Legendrian link $\Lambda \subset \mathbb{R}^3$ consisting of two unlinked standard unknots, with the surgery arc $\eta_m \subset \mathbb{R}^3$  as shown in Figure \ref{fig:handle_decomposition_odd} or \ref{fig:handle_decomposition_even}. In particular, the ambient Legendrian surgery produces the Legendrian knot $\Lambda_{\eta_m} = \Lambda_m$. Then consider the Legendrian surgery arc $\eta^0_m$ for the same Legendrian link $\Lambda$ depicted in Figure \ref{fig:unknotodd}, for which $\Lambda_{\eta^0_m}$ clearly is Legendrian isotopic to the standard unknot $U$. Note that $\eta^0_m$ and $\eta_m$ both are stabilised and Legendrian isotopic in $\mathbb{R}^3$ relative boundary. However, any isotopy must clearly intersect $\Lambda$.

In the case when $n$ is odd, we can apply Corollary \ref{cor:legiso1} directly to get the sought isotopy between the spun of $\eta_m$ and the spun of $\eta^0_m$. Lemma \ref{lem:isosurgery} then gives the sought Legendrian isotopy between the spuns.

In the case when $n$ is even, we can use the contact isotopy depicted in Figure \ref{fig:move_even} in order to verify the existence of a Legendrian isotopy that takes the arc $\eta_m$ to $\eta^0_m$, and for which the assumptions of Proposition \ref{charact_rel} are met. Indeed, the obvious isotopy $\eta_t$ described in Figure \ref{fig:move_even2} intersects $\Lambda$ twice, but with opposite signs of the corresponding double points on the induced trace cobordism. Again, Corollary \ref{cor:legiso1} shows the claim.

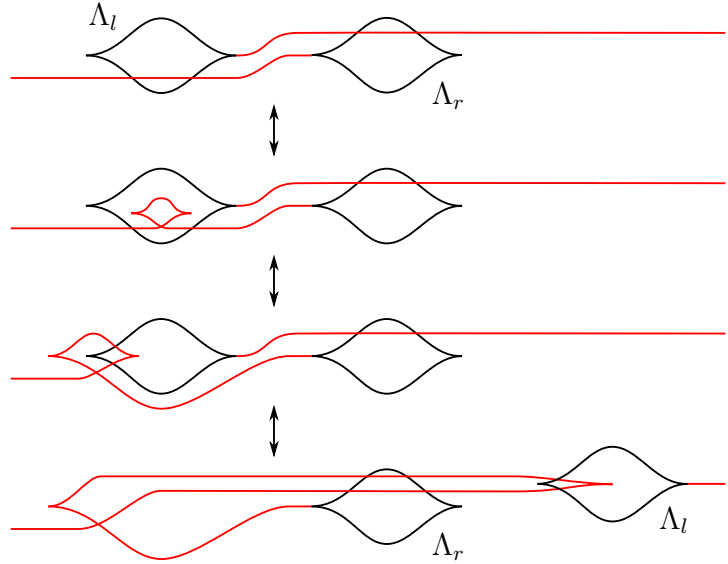
\begin{figure}[h!]
\begin{center}
\labellist
\pinlabel $\Lambda_l$ at 35 205
\pinlabel $\Lambda_r$ at 165 175
\pinlabel $\Lambda_l$ at 35 150
\pinlabel $\Lambda_r$ at 165 120
\pinlabel $\Lambda_l$ at 45 98
\pinlabel $\Lambda_r$ at 165 65
\pinlabel $\Lambda_l$ at 250 15
\pinlabel $\Lambda_r$ at 165 5
\endlabellist
\includegraphics[scale=1]{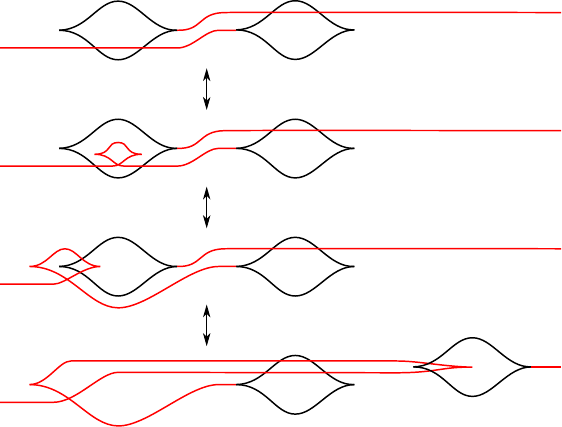}
\caption{The configurations showed on the top and second from bottom are contact isotopic.}
\label{fig:move_even}
\end{center}
\end{figure}
\begin{figure}[h!]
\begin{center}
\labellist
\pinlabel $\Lambda_l$ at 250 80
\pinlabel $\Lambda_r$ at 165 70
\pinlabel $\Lambda_l$ at 250 10
\pinlabel $\Lambda_r$ at 80 0
\endlabellist
\includegraphics[scale=1]{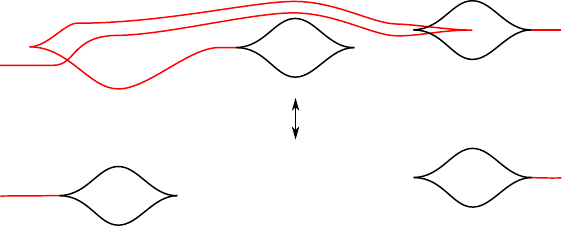}
\caption{The bottom configuration in Figure \ref{fig:move_even} can be deformed to the top figuration shown here by a Legendrian isotopy of the red arc that intersects $\Lambda_r$, i.e.~unlinking the red arc from $\Lambda_r$. The corresponding immersed trace cobordism has two double points of opposite signs. Finally, after this unlinking, one can apply a global contact isotopy to make both the surgery arcs and Legendrian unknots unlinked as shown at the bottom.}
\label{fig:move_even2}
\end{center}
\end{figure}

\subsection{Second strategy: Spinning the pair of cancelling Weinstein handles} 
\label{sec_str}

We apply Proposition \ref{prop:surgeries}  to the standard Lagrangian handle-attachment cobordism whose negative  boundary is the two unlinked unknots $\Lambda$ and positive boundary is the Legendrian $\Lambda_m$, i.e.~the result $\Lambda_{\eta_m}$ of an ambient Legendrian surgery on $\Lambda$ along the isotropic arc $\eta_m$. The result is a Weinstein structure on $\R \times \R^3$ consisting of a Weinstein one and two-handle in cancelling position for which the Lagrangian handle-attachment cobordism is tangent to the Liouville flow.

In particular Proposition \ref{prop:surgeries} gives us a Legendrian $\tilde{\Lambda}_{\eta_{\color{black}m}} \subset \left(\R^3\right)_{sc}$ where the latter contact manifold is the result of a Weinstein $0$-surgery on $\R^3$, and such that $\tilde{\Lambda}_{\eta_m}$ becomes $\Lambda_{\eta_m}$ after a Weinstein 1-surgery on a knot $S \subset \left(\R^3\right)_{sc} \setminus \tilde{\Lambda}_{\eta_m}$ that cancels the $0$-surgery. The attaching $0$-sphere and cancelling $1$-sphere, as well as the knot $\tilde{\Lambda}_{\eta_m}$, are obtained from $\Lambda$ and $\eta_m$ by the recipe shown in Figure \ref{fig:darbouxball}; \color{black}we refer to Figure \ref{fig:handle_decomposition_even} for $\eta_m$.\color{black}
 
There exists a stabilised Legendrian knot $S_0 \subset \left(\R^3\right)_{sc} \setminus \Lambda_{\eta_m}$ which also is in cancelling position, but for which $\Lambda_{\eta_m}$ becomes the standard unknot $U$ after cancelling \color{black} the 1-handle with the 2-handle attached along $S_0$. This attaching sphere can be constructed analogously to $S$, but using the surgery arc $\eta^0_m$ shown in Figure \ref{fig:unknotodd} instead of $\eta_m$ shown in Figure \ref{fig:handle_decomposition_even}. Here we again \color{black} follow the recipe in Figure \ref{fig:darbouxball}.

\color{black}
\subsubsection{Digresson: Generalised Weinstein handle decompositions}

The product of two Weinstein manifolds are naturally equipped with a handle-body structure where the handles of the product corresponds to products of handles. In particular, the number of handles for the induced handle-decomposition is the product of the number of handles of each factor. In the case when one of the factors is the cotangent bundle $T^*M$ of a closed smooth manifold there is a more economic decomposition, if one allows a generalisation of the Weinstein handle decomposition.

First recall that the standard Weinstein $k$-handle, as introduced by Weinstein in \cite{Weinstein}, in dimension $2n$ with $n \ge k$ is given by $T^*D^k \times \C^{n-k}$ with the standard symplectic structure. This handle can be attached to a Liouville domain by, roughly speaking, gluing the $k-1$-sphere
$$\partial 0_{D^k} \times \{0\} \subset T^*D^k \times \C^{n-k}$$
to a framed isotropic $k-1$ sphere in the boundary of a Liouville domain. The latter framed isotropic embedding is called the \textbf{isotropic attaching sphere}. A Weinstein manifold can be characterised as a Liouville domain that admits a handle-body decomposition consisting of standard Weinstein handles.

A \textbf{generalised Weinstein $k$-handle} is given by $T^*N^k \times \C^{n-k}$ where $D^k$ has been replaced by a general compact and smooth $k$-dimensional manifold $N^k$ with non-empty boundary. The generalised handles are similarly attached to a Liouville domain by gluing
$$\partial 0_N \times \{0\} \subset T^*M^k \times \C^{n-k}$$
along a framed isotropic embedding of $\partial N$ in the the contact boundary of the Liouville domain. We refer to the latter framed isotropic embedding as the \textbf{isotropic attaching manifold}; see \cite{Dimitroglou:Refined} and \cite{Husin} for more details on particular generalised handles. Those Liouville manifolds that admit a handle-body structure consisting of generalised Weinstein handles are obviously also Weinstien manifolds. By a \textbf{generalised Weinstein handle body} we mean a decomposition of a Weinstein manifold into generalised Weinstein handles.

A particular set of examples that admit natural generalised Weinstein handle decompositions is the product $X \times T^*M^m$ of a Weinstein manifold $X$ and a cotangent bundle $T^*M$ of a closed smooth $m$-dimensional manifold. Namely, this product has a simple decomposition into generalised Weinstein handles induced by any Weinstein handle decomposition of $X$. Namely, each Weinstein handle $T^*D^k \times \C^{n-k}$ in the handle-decomposition of $X$ gives rise to a generalised Weinstein handle $T^*(D^k \times M) \times \C^{n-k}$ in the generalised Weinstein handle-body decomposition of $X \times T^*M$. The topology of each handle in this case is thus $N^{k+m}=D^k \times M^m$.\color{black}

\subsubsection{A generalised Weinstein handle decomposition of the symplectisation $\R \times (\R^3 \times T^*S^n)$.}

\color{black} In our situation we are interested in taking the product \color{black} with $T^*S^n$. This results in a generalised Weinstein handle decomposition of \color{black}the symplectisation $\R \times (\R^3 \times T^*S^n)$ obtained in the following manner. First, consider a Weinstein 1-handle attached on $\R^3$ and take the product with $T^*S^n$. This leads to a product Weinstein manifold with boundary  $\left(\R^3\right)_{sc} \times T^*S^n.$ The latter contact manifold is obtained from $\R^3 \times T^*S^n$ by a pair of subcritical Weinstein surgeries along isotropic $n$-spheres. \color{black} Since the Weinstein 1-handle attachment that produces $\left(\R^3\right)_{sc}$ can be cancelled by a Weinstein 2-handle attached along $S \subset \left(\R^3\right)_{sc}$, taking a product of the latter handle and $T^*S^n$ then yields $\R^3 \times T^*S^n$ obtained from $\left(\R^3\right)_{sc} \times T^*S^n$ \color{black} by a generalised Weinstein surgery along the Legendrian $S \times 0_{S^n}$. Note that the latter generalised surgery corresponds to attaching a generalised Weinstein handle of the form $T^*(D^2 \times S^n)$. Note that such a handle-attachment, as well as the corresponding contact surgery, is completely determined by a choice of Legendrian embedding of $D^2 \times S^n$.

Finally we can apply Corollary \ref{cor:legiso2} to the two attaching manifolds $S_0 \times 0_{S^n}$ and $S \times 0_{S^n}$ to deduce the claim that $\Lambda_m \times 0_{S^n}$ and $U \times 0_{S^n}$ are Legendrian isotopic.

Alternatively, one can invoke the $h$-principle for flexible Lagrangian cobordisms \cite[Theorem 4.2]{EliashbergGanatraLzarev20} to deduce that there is a symplectomorphism of the symplectisation of $\R^3 \times T^*S^n$ that identifies the standard Lagrangian fillings of $U \times 0_{S^n}$ with the one of $\Lambda_m \times 0_{S^n}$ (given by the product of the one in Figure \ref{fig:filling of pretzel}). The Legendrian isotopy is then constructed by applying the below Lemma \ref{lem:isotopy} to the contactomorphism of the contact boundary that is induced by the symplectomorphism.

\begin{lem}
\label{lem:isotopy}
Assume that $(Y,\xi=\ker \alpha)$ is a contact manifold for which any proper compact subset $K \subsetneq Y$ admits a displacing contact isotopy $\phi_t \in \mathrm{Cont}(Y,\xi)$, by which we mean that $\phi_1(K) \cap K = \emptyset$. (This is e.g.~the case for $Y=\R^{2m+1}$ or the standard contact $S^{2m+1}$.) Then, if $\Lambda \subset (Y \times T^*S^n,\alpha+p\,d\theta)$ is a closed Legendrian, the image $\Psi(\Lambda)$ is Legendrian isotopic to $\Lambda$ whenever $\Psi \in \mathrm{Cont}(Y\times T^*S^n)$ is a contactomorphism with support contained in $K \times T^*S^n$ for some proper compact subset $K \subsetneq Y$. 
\end{lem}
\begin{proof}
Any contact isotopy $\phi_t \in \mathrm{Cont}(Y,\xi)$ lifts to a contact isotopy
\begin{gather*}\Phi_t\in \mathrm{Cont}(Y \times T^*S^n),\\
\Phi_t(y,\theta,p)=(\phi_t(y),\theta,p\cdot e^{g_t(y)}) 
\end{gather*}
where $g_t \colon Y \to \R$ is determined by $\phi_t^*\alpha=e^{g_t}\alpha$. For any contactomorphism $\Psi \in \mathrm{Cont}(Y \times T^*S^n)$ supported in $K \times T^*S^n$ with $K \subsetneq Y$ proper and compact, we can find a contact isotopy of the above form for which $\Phi_1$ displaces the support of $\Psi$. The image $\Psi(\Lambda)$ is thus Legendrian isotopic to $\Psi \circ \Phi_1(\Lambda)=\Phi_1(\Lambda)$ which, in turn, is Legendrian isotopic to $\Phi_0(\Lambda)=\Lambda$.
\end{proof}

\begin{figure}[t]
\begin{center}
\labellist
\pinlabel $\color{red}\eta_m$ at 10 48
\endlabellist
\includegraphics[scale=0.75]{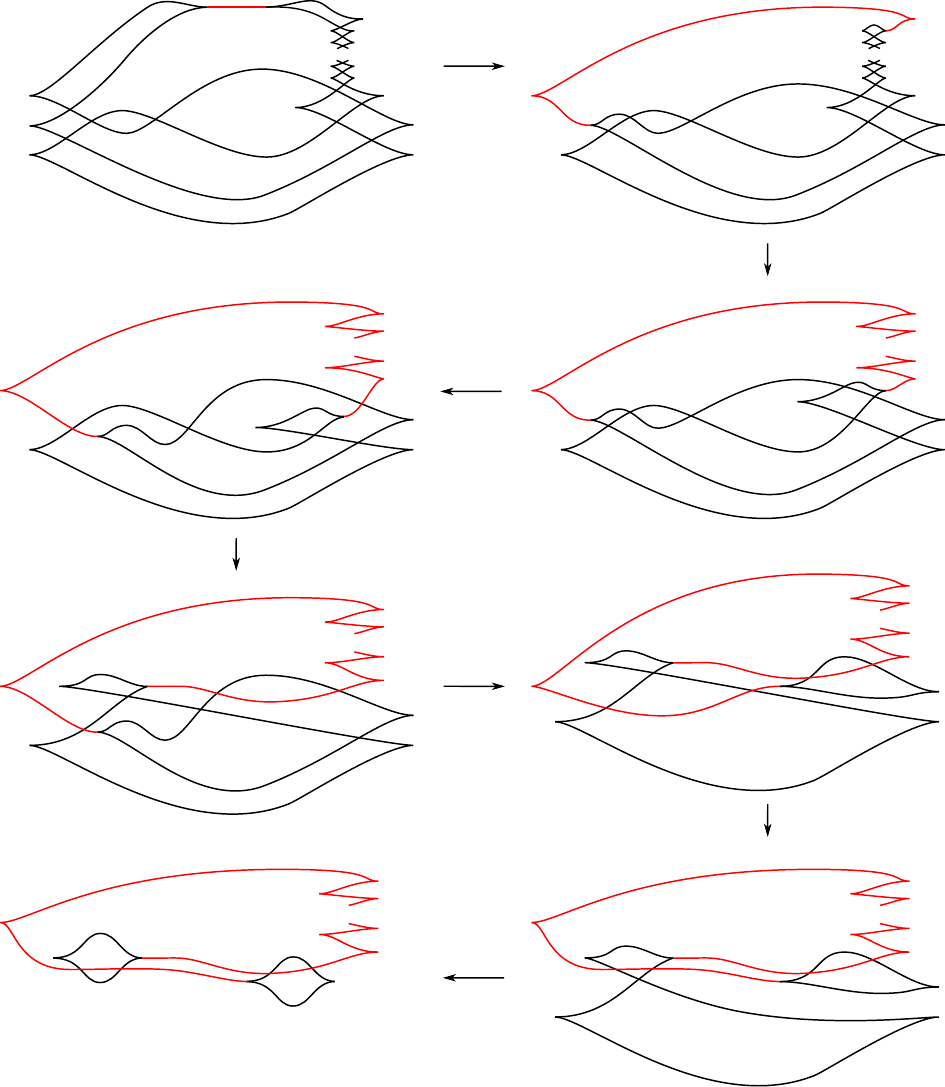}
\caption{$\Lambda_m$, $m>3$ odd, can be produced by an ambient Legendrian surgery along the surgery arc $\eta_m$ as shown.}
\label{fig:handle_decomposition_odd}
\end{center}
\end{figure}

\begin{figure}[t]
\label{mmeven}
\begin{center}
\labellist
\pinlabel $\color{red}\eta_m$ at 350 80
\endlabellist
\includegraphics[scale=0.75]{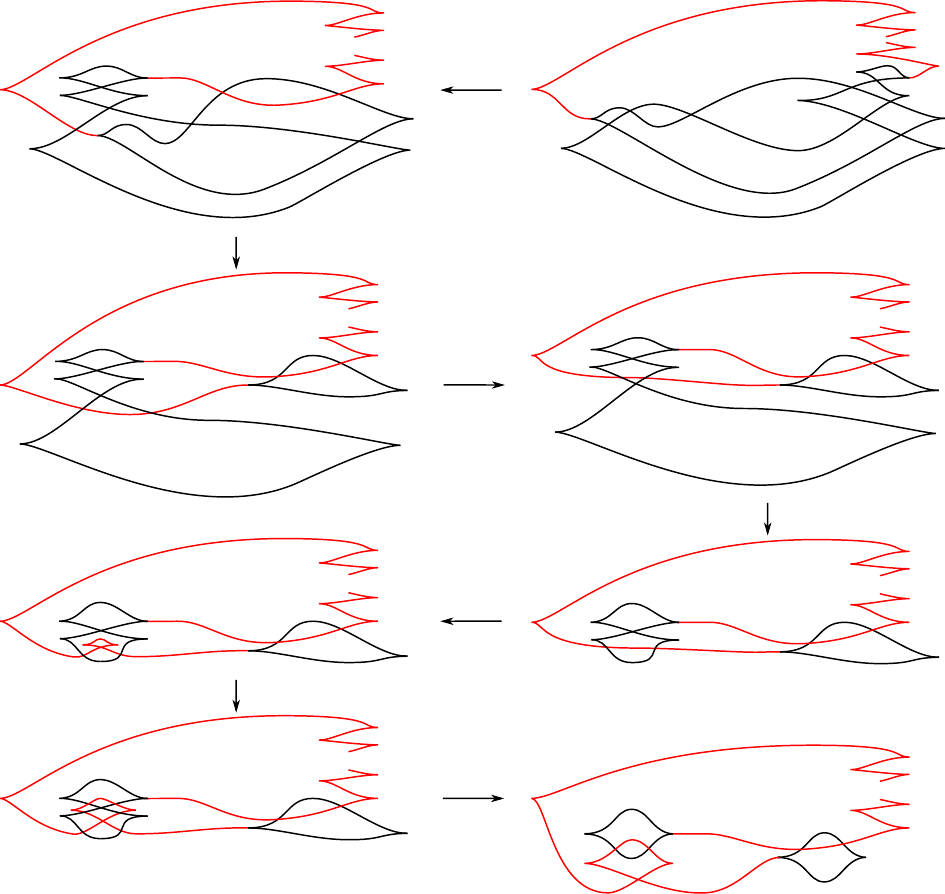}
\caption{$\Lambda_m$, $m>3$ even, can be produced by an ambient Legendrian surgery along the surgery arc $\eta_m$ as shown. See Figure \ref{fig:handle_decomposition_odd} for the first steps.}
\label{fig:handle_decomposition_even}
\end{center}
\end{figure}

\begin{figure}[t]
\begin{center}
\labellist
\pinlabel $\color{red}\eta^0_m$ at 4 10
\endlabellist
\includegraphics[scale=0.75]{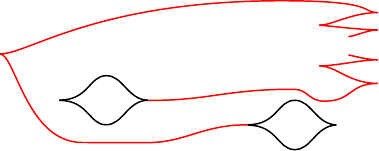}
\caption{Performing the Legendrian ambient surgery on two unlinked unknots, whose front projections moreover are disjoint, along an arc $\eta^0_m$ that intersects the two unknots precisely in their boundary points, produces the standard unknot $U$.}
\label{fig:unknotodd}
\end{center}
\end{figure}

\section{Constructing non-regular Lagrangian concordances}
The existence of non-regular exact Lagrangian cobordisms were shown by Eliashberg and Murphy in \cite{LagCaps} in high dimensions and Lin \cite{Lin} in dimension four. Namely, these constructions produce exact Lagrangian cobordisms with an empty positive end inside a symplectisation; elementary topological considerations imply that such a cobordism cannot be regular. From this perspective, the examples that we construct here are different, since they have a non-empty positive end and the simplest possible topology among Lagrangian cobordisms; they are concordances.

First we need the following generalisation of \cite[Theorem 3.2]{CornwellNgSivek16}.
\begin{thm}
\label{thm:regularconc}
Let $C \subset [-T,T] \times S^3$ be a smooth concordance that coincides with the cylinder over a knot $K_\pm$ near the boundary, and such that $[-T,T] \times S^3$ can be obtained by attaching $0,1,$ and $2$-handles onto a neighborhood of
$$(\{-T\} \times S^3) \cup C \subset [-T,T]\times S^3.$$
If $K_+$ is the unknot, then so is $K_-$. 
\end{thm}
\begin{proof}
The proof is an adaptation of the proof of \cite[Theorem 3.2]{CornwellNgSivek16}, which showed that such a concordance $C$ cannot be ribbon, i.e.~the symplectisation coordinate must have a local maximum somewhere.

Consider the handle-decomposition of $[-T,T] \times S^3$ in which all handles are attached onto a neighbourhood $(\{-T\} \times S^3) \cup C$, and for which there are no $3$- or $4$-handles. We may also assume that there are no $0$-handles, and that the $2$-handles are attached after the $1$-handles. Since $H_*([-T,T]\times S^3,\{-T\} \times S^3)=0$, it follows that the number of $1$-handles is equal to the number of $2$-handles; denote this number by $h$.

Then we consider the double cover of $[-T,T] \times S^3$ branched over the concordance $C$. This branched cover thus has a handle decomposition for which there are $2h$ handles in the cover of index one as well as of index two. This branched double cover is a cobordism from $\Sigma(K_-)$ to $\Sigma(K_+)$, where $\Sigma(K_\pm)$ denotes the double cover of $S^3$ branched along the knot $K_\pm$. Note that $\Sigma(K_+)=S^3$ since $K_+$ is the unknot.

Attaching a four-handle to the end $\Sigma(K_+)=S^3$ we obtain a cobordism $X$ from $\Sigma(K_-)$ to $\emptyset$ which is simply connected. Indeed, turning it upside down, we have a cobordism with only $0$-, $2$-, and $3$-handles.

If we then consider $X$ as obtained by attaching $1$, $2$, and $4$-handles on $\Sigma(K_-)$, this means that the trivial group can obtained from $\pi_1(\Sigma(K_-))$ by adding $2h$ new generators (coming from the one-handles) and $2h$ new relations (coming from the two-handles).

The rest of the proof is the same as the argument in \cite[Theorem 3.2]{CornwellNgSivek16}. In particular, this argument uses the adaptation \cite[Theorem 3.1]{CornwellNgSivek16} of a result of Kronheimer and Mrowka that shows that that $\pi_1(\Sigma(K_-))$ admits a non-trivial representation $\rho$ in $SO(3)$ if and only if $K_-$ is not the unknot, while the argument in \cite[Section I.3]{Kirby}
shows that there can be no such representation under the assumption that $2h$ new generators and relations kill the group $\pi_1(\Sigma(K_-))$. 

We finish by recalling the latter argument. The group $G_0 \coloneqq \pi_1(\Sigma(K_-))/\ker \rho$ is a non-trivial group by the assumption that $\rho$ is non-trivial. This quotient group can be taken to have the same generators as $\pi_1(\Sigma(K_-))$ (in particular it is finitely generated). In addition, it admits a faithful non-trivial representation into $SO(3)$. Adjoining $2h$ new generators, i.e.~considering the free product $G_0 \star \F_{2h}$ of $G_0$ and the free group on $2h$ generators, together with the $2h$ relations as above, we obtain a group $G$. \color{black}This group is the quotient of the free product $\pi_1(\Sigma(K_-)) \star \F_{2h}$
by the normal subgroup generated by the $2h$ relations together with the relations from $\ker \rho$. 
\color{black}
In particular, $G$ is the quotient of the trivial group $\pi_1(X)$ and, thus, $G$ is itself trivial. 
Finally we can use \cite[Theorem 3]{GerstenhaberRothaus}
to show that $G_0$ embeds in $G$, which means that $G_0$ is trivial, and thus leading to the sought contradiction. (Note that the condition on the determinant of the $2h$ relations needed in the cited theorem \color{black} is \color{black} satisfied here by the argument above, by which  $H_1(\Sigma(K_-))$ is killed inside $H_1(X)$ by these $2h$ relations in the abelianisation.)

\end{proof}
\begin{remark}
It is unknown to the authors whether a regular Lagrangian cobordism is also decomposable. A decomposable Lagrangian cobordism $L \subset [-T,T] \times S^3$ is regular by a Weinstein handle decomposition of $[-T,T] \times S^3$ \color{black} of a very particular type\color{black}, in which the handle decomposition of the neighbourhood $T^*L$ and the handles attached to $\{-T\} \times S^3 \cup T^*L$ are in cancelling position. It is not clear if a general set of Weinstein handle attachment on $\{-T\} \times S^3 \cup T^*L$ that produces $[-T,T] \times S^3$ can be deformed into such a position.
\end{remark}

In recent work \cite{Dimitroglou:Approximations} by the first author it is shown that totally real concordances can be deformed \color{black} in order to become \color{black} Lagrangian after adding sufficiently many stabilisations to the Legendrian ends. \color{black} To formulate the result we need to recall the following definition. An almost complex structure on the symplectisation $(\R_t \times Y,d(e^t\,\alpha))$ of a contact manifold $(Y,\alpha)$ is called \textbf{cylindrical} if
\begin{itemize}
\item $J$ is invariant under translation of the $t$-coordinate;
  \item $J$ preserves the contact planes $\ker \alpha \subset TY$  in each level-set $\{t\}\times Y$ and is compatible with $d\alpha$ inside the same, i.e.~$d\alpha(\cdot,J\cdot)$ defines a Riemannian metric on $\ker \alpha$; and
  \item $J \partial_t=R_\alpha$ where $R_\alpha$ is the Reeb vector field.
  \end{itemize} \color{black}
\begin{thm}[\cite{Dimitroglou:Approximations}]
\label{thm:flexibility}
Assume that \color{black} for a given choice of cylindrical almost complex structure \color{black} on $(\R \times S^3,{\color{black} d(e^t\,\alpha_{st})})$ \color{black} there exists a totally real concordance from a Legendrian knot $\Lambda_-'$ to $\Lambda_+'$ (which according to Appendix \ref{appendix_totally_real} implies that $\Lambda_\pm'$ are smoothly concordant, and with the same ${\tt tb}$ and ${\tt rot}$). Then this concordance can be deformed to a Lagrangian concordance from $\Lambda_-$ to $\Lambda_+$ by a smooth isotopy, where both $\Lambda_\pm$ are obtained from $\Lambda_\pm'$ by performing a sufficiently large number $k\gg 0$ of stabilisations of both signs.
\end{thm}

We will construct totally real concordances by applying the $h$-principle for totally real embeddings to Lagrangian concordances whose symplectisation coordinate has been reversed. Note that Lagrangian condition is not necessarily preserved after reversing the symplectisation coordinate, except at the points where the Lagrangian is \emph{cylindrical}. In order to show that the tangent spaces can be homotoped to totally real planes after reversing the concordance, the following result is crucial.
\begin{prop}
  \label{prop:regularhomotopy}
 Any Lagrangian concordance
  $$C\subset (\R_t \times \R^3_{xyz},d(e^t(dz-y\,dx)))$$
  from a Legendrian knot $\Lambda_-$ to a Legendrian knot $\Lambda_+$ admits a regular Lagrangian homotopy which is fixed at the negative end, and cylindrical at the positive end (but not necessarily embedded), to the trivial Lagrangian concordance $ \R \times \Lambda_-$.
\end{prop}
\begin{proof}
  Since the rotation numbers  of $\Lambda_+$ and $\Lambda_-$ agree, there is a Legendrian regular homotopy from $\Lambda_+$ to $\Lambda_-$ by the $h$-principle for Legendrian immersions \cite[24.1.3]{EliashbergMishachevCielebak}. Using a cylindrical extension of this Legendrian regular homotopy, one readily produces a regular (non-compactly supported) Lagrangian homotopy from $C$ to an immersed Lagrangian concordance $C'$ from $\Lambda_-$ to $\Lambda_-$ that is fixed at the negative end, and which is cylindrical at the positive end through the entire deformation. 
This construction can be done by generating the regular homotopy of the corresponding cylindrical  Lagrangian immersion by a multi-valued Hamiltonian on the symplectisation, which can be suitably cut-off so that it becomes supported on the positive cylindrical end. To that end, recall that Legendrian isotopies can be extended to a contact isotopy that is generated by a non-autonomous contact Hamiltonian $H \colon \R^3 \to \R$, and that the cylindrical lift of this contact isotopy to the symplectisation is generated by the Hamiltonian $e^tH \colon \R \times \R^3 \to \R$. Of course, since we consider Legendrian regular homotopies rather than isotopies, the Hamiltonians are necessarily multi-valued near the intersecting branches, but these multi-valued Hamiltonians still generate  a well-defined isotopy of each branch separately, i.e.~a regular Lagrangian homotopy.

  What remains is to show that $C'$ is compactly regular Lagrangian homotopic to the trivial cobordism $\R \times \Lambda_-$ by applying the $h$-principle for Lagrangian immersions \cite[24.3.1]{EliashbergMishachevCielebak}. However, to that end, it will be necessary to first rotate the positive end a number of times in order to correct for the differences of the Maslov classes of the positive and negative ends.  We refer the reader to \cite[Remark 4.5(1)]{CDRGGokova} for a similar construction. We proceed to outline the details.

Take a path $\gamma$ from the negative to the positive end of $C'$ that coincides with $\R \times \{\operatorname{pt}\}$ outside of a compact subset $([-T,T]\times \R^3)\cap C'$, for some choice of point $\operatorname{pt} \in \Lambda_-$.

Since there is a canonical identification of the tangent planes of $C'$ at the pair of points of $\gamma$ given by $(t', \operatorname{pt}), (-t', \operatorname{pt})\in \R\times \R^3$ , where $t'>T$, there is a well-defined Maslov class of Lagrangian tangent planes along the line. When this Maslov class is non-zero, we can readily apply a Hamiltonian isotopy of the symplectisation $\R \times \R^3$ that fixes the negative end, and is a lift under $\R^3_{xyz} \to \R^2_{xy}$ of a rotation at the positive end, so that the following is satisfied: The concordance $C'$ is deformed to a concordance $C''$ from $\Lambda_-$ to $\Lambda_-$, and the path $\gamma$ is deformed to a path $\tilde{\gamma}$, where the latter path coincides with $\R \times \{\operatorname{pt}\}$ outside of a compact subset, along which the Lagrangian tangent planes has a vanishing Maslov class.

Using the fact that the above path $\tilde{\gamma}$ has vanishing Maslov class, we can construct a compactly supported regular Lagrangian homotopy that normalises the path, so that the concordance becomes cylindrical near all of $\R \times \{\operatorname{pt}\} \subset \R \times \R^3$. Finally, since the second homotopy group of oriented Lagrangian planes vanishes, i.e.~$\pi_2(U(2))=0$, it follows that the entire immersed Lagrangian concordance is compactly supported regular Lagrangian homotopic to the trivial cylinder over $\Lambda_-$, as sought.
\end{proof}
\color{black}
\begin{prop}
\label{prop:reversible}

For any decomposable Lagrangian concordance
$$C\subset (\R_t \times \color{black}\R^3_{xyz},d(e^t(dz-y\,dx)))$$
from a Legendrian knot $\Lambda_-$ to a Legendrian knot $\Lambda_+$, the reversed cobordism
$$ C' \coloneqq \{(-t,x,y,z) \in \R_t \times \R^3_{xyz}: (t,x,y,z) \in C\}$$
admits a compactly supported smooth isotopy to a cobordism that is totally real, for any given choice of cylindrical almost complex structure.
\end{prop}
\begin{proof}
   
  Outside of a sufficiently large compact subset that contains the non-cylindrical part of $C$, the reversed cobordism $C'$ is already Lagrangian, and thus totally real for any compatible almost complex structure. We will analyse the homotopy class of the tangent planes of $C'$ inside $\R \times \R^3$, in order to show that the map is compactly supported homotopic to one which takes values in Lagrangian tangent planes; the latter are in particular totally real for any compatible almost complex structure. The existence of the smooth isotopy is then a consequence of the $h$-principle for totally real embeddings relative to the boundary \cite[27.4.1 and Example 2]{EliashbergMishachevCielebak}.

In order to show the existence of the homotopy of the Gauss map we argue as follows. First, by the existence of the regular Lagrangian homotopy produced by Proposition \ref{prop:regularhomotopy}, the Lagrangian Gauss map of $C$ admits a homotopy to a map that takes values in Lagrangian planes that are everywhere tangent to $\partial_t$, i.e.~that takes values in \emph{cylindrical} Lagrangian tangent planes. Note that, by construction, this regular homotopy is fixed on the negative cylindrical end of $C$, and cylindrical (but not fixed) at the positive cylindrical end of $C$. However, we want to find a compactly supported homotopy of the Gauss map to one that takes values in cylindrical Lagrangian tangent planes. In order to obtain this, we can simply perform an interpolating homotopy at the positive end of $C$ (here the entire original homotopy is through cylindrical Lagrangian tangent planes by construction), to make it fixed outside of a sufficiently large compact subset. Thus we have managed to create the sought compactly supported homotopy of the Lagrangian Gauss map.\color{black}

 Note that the smooth involution of the symplectisation $\R \times \R^3$ given by $(t,(x,y,z)) \mapsto (-t,(x,y,z))$  preserves those two-planes that are tangent to $\partial_t$, e.g.~the cylindrical Lagrangian tangent planes. Since $C$ has a Gauss map that is compactly supported homotopic to a map of cylindrical Lagrangian planes, we conclude that the tangent map of $C'$ is also compactly supported homotopic to a map of Lagrangian tangent planes. This shows that the $h$-principle for totally real embeddings can be applied to $C'$ as sought.
\end{proof}

\begin{remark}
Note that totally real concordances define a relation \color{black} on the set of Legendrian isotopy classes of Legendrian knots \color{black} which is obviously reflexive and transitive.  From Proposition \ref{prop:reversible} it follows that this relation is not anti-symmetric, and hence totally real concordances do not define a partial order.
\end{remark}

\begin{thm} 
\label{nonregular_concordance}
Let $\Lambda$ be a Legendrian knot that admits a decomposable Lagrangian disc filling. Then there exists an exact Lagrangian concordance $C \subset (\R_t \times \R^3,d(e^t\alpha_0))$ with concave Legendrian boundary $\Lambda_-$ obtained by $k$-fold positive and negative stabilisation of $\Lambda$ for $k \gg 0$ sufficiently large, and a convex Legendrian boundary $\Lambda_+$ which is the standard Legendrian unknot $U$ stabilised the same number of times. \end{thm}
\begin{example}
Such Lagrangian concordances thus exists from sufficiently stabilised versions of $\Lambda_m$, e.g.~the $\overline{9_{46}}$-knot $\Lambda_3$. See Figures  \ref{fig:handle_decomposition_odd} and \ref{fig:handle_decomposition_even} for the construction of the Lagrangian slice discs. 
\end{example}

\begin{remark}
By Theorem \ref{thm:regularconc} the Lagrangian concordances produced by Corollary \ref{nonregular_concordance} are never regular. Therefore, from the result of Conway, Etnyre and Tosun \cite{ConwayEtnyreTosun21} \color{black}(or, alternatively, from our generalisation Proposition \ref{prop:surgeries})\color{black}, it follows that these Lagrangian concordances also are not decomposable in the sense of Chantraine \cite{Chantraine12}.
\end{remark}

\begin{proof}
A decomposable Lagrangian disc filling can be seen as a standard disc filling of a standard Legendrian unknot $U$ concatenated with a concordance from $U$ to $\Lambda$. We apply Proposition \ref{prop:reversible} in order to produce a totally real concordance from $\Lambda$ to $U$. Theorem \ref{thm:flexibility} then gives us the sought Lagrangian concordance, by which we first need to add sufficiently many positive and negative \color{black} stabilisations to $U$ and to $\Lambda$. \color{black}
\end{proof}

\begin{remark}
\label{spinconcordance}
Consider the spun  $\Sigma_{S^1}C$ of the non-regular concordance $C$ from Corollary \ref{nonregular_concordance}. Note that since by the construction of $C$ the positive and negative ends of $C$ are stabilized Legendrian knots, the positive and negative ends of $\Sigma_{S^1}C$  are loose Legendrian tori by \cite{Golovko14}. Then using the fact that the positive and negative ends of $\Sigma_{S^1}C$ have the same classical Legendrian invariants, they are Legendrian isotopic by the h-principle for loose Legendrian submanifolds, see \cite{Murphy}. From the expected parametric version of h-principle   for exact Lagrangian cobordisms with loose negative ends \cite{EliashbergLazarevMurphy} it should follow that $\Sigma_{S^1}C$ is Hamiltonian isotopic to the trivial cylinder. 
Finally, following \cite{Golovko14} we see that the same can be said about $\Sigma_{S^{k_1}}\dots \Sigma_{S^{k_r}}\Sigma_{S^1}C$ for $k_1,\dots,k_r\geq 1$.

\end{remark}

\begin{appendix}
\section{Classical obstructions to the existence of totally real cobordisms between Legendrian knots}
\label{appendix_totally_real}

Here we prove the direct generalization of the result of Chantraine \cite{Chantraine10} that provides the classical obstructions   to the existence of totally real cobordisms between Legendrian knots. 
\begin{thm}
\label{totallyreal}
If there is a totally real orientable cobordism $L$ from \color{black} null-homologous Legendrian knot $\Lambda_-$ with a Seifert surface $S$ \color{black} to Legendrian knot $\Lambda_+$, then 
\begin{align*}
&tb(\Lambda_+)-tb(\Lambda_-)={\color{black}-\chi(L)},\\
&r(\Lambda_-, [S])=r(\Lambda_+, [S\cup L]).
\end{align*}
\end{thm}
\begin{proof}
The proof is completely analogous to the case when the cobordism is Lagrangian. Namely, we perturb the symplectisation coordinate in a compact subset so that it restricts to a Morse function $h\colon \color{black}L\color{black} \to \R$ on the cobordism. Then we consider a push-off by $J(\nabla h)$ for an almost complex structure that is cylindrical outside of a compact subset, compatible with the symplectic structure, and for which the cobordism is totally-real. We can assume that $\nabla h=\partial_t$ is satisfied outside of a compact subset. It follows that the signed count of intersection points are precisely $-\chi(\color{black}L)$ (as in the Lagrangian case). The rest follows in the same manner. (Recall that the Thurston-Bennequin invariant is the self-intersection number computed by the push-off of a Legendrian along the vector field $J\partial_t$.)
\end{proof}
\end{appendix}


\begin{thebibliography}{}


\bibitem{BourgeoisSabloffTraynor15}
F.~Bourgeois, J.~Sabloff and L.~Traynor, {\em Lagrangian cobordisms via generating families:
Construction and geography}, Algebraic and Geometric Topology 15 (2015) 2439--2477

\bibitem{CasalsMurphy}
R.~Casals and E.~Murphy, {\em Contact topology from the loose viewpoint.} Proceedings of the Gökova Geometry-Topology Conference 2015, 81--115. Gökova Geometry/Topology Conference (GGT), Gökova 2016.

\bibitem{Chantraine10}
B. ~Chantraine, {\em Lagrangian concordance of Legendrian knots},
Algebr. Geom. Topol. {\bf 10} (2010), 63--85.

\bibitem{Chantraine12}
  B. ~Chantraine, {\em Some non-collarable slices of Lagrangian surfaces}, Bull. Lond. Math. Soc. {\bf 44} (2012),  981--987.

  
\bibitem{CDRGGokova}
B.~Chantraine, G.~Dimitroglou Rizell, P.~Ghiggini, and R.~Golovko, {\em Floer homology and Lagrangian concordance}, Proc. Gökova Geom. Topol. Conf. 2014 (2015), 76--113.


\bibitem{ChantraineLegout22}
B.~Chantraine and N.~Legout, {\em Doubly slice knots and obstruction to Lagrangian concordance}, Comptes Rendus. Mathematique, Volume 361 (2023),  1605--1609.

\bibitem{ConwayEtnyreTosun21}
J. ~Conway, J. ~Etnyre and B. ~Tosun, {\em Symplectic fillings, contact surgeries, and Lagrangian disks}, International Mathematics Research Notices, Volume 2021, Issue 8, April 2021, 6020–6050.

\bibitem{CornwellNgSivek16}
C.~Cornwell, L.~Ng and S.~Sivek, {\em Obstructions to Lagrangian concordance},
Algebr. Geom. Topol. 16(2): 797-824 (2016).

\bibitem{Dimitroglou:Approximations}
G.~Dimitroglou~Rizell. {\em Lagrangian Approximation of Totally Real Concordances}, preprint 2024, arXiv:2408.16614.

\bibitem{Dimitroglou:Ambient}
G.~Dimitroglou~Rizell. {\em Legendrian ambient surgery and {L}egendrian contact homology}, J. Symplectic Geom. 14(3):811--901 (2016).

\bibitem{Dimitroglou:Refined}
G.~Dimitroglou Rizell, T.~Ekholm, and D.~Tonkonog,
\emph{Refined disk potentials for immersed Lagrangian surfaces},
\textit{J. Differential Geom.} \textbf{121}(3) (2022), 459--539,
\href{https://doi.org/10.4310/jdg/1664378618}{doi:10.4310/jdg/1664378618}.\color{black}

\bibitem{DimitroglouRizellGolovko14}
G.~Dimitroglou Rizell and R.~Golovko, {\em On homological rigidity and flexibility of exact Lagrangian endocobordisms}, International Journal of Mathematics, Vol. 25, No. 10, 1450098 (2014).

\bibitem{DimitroglouRizellGolovko19}
G. ~Dimitroglou Rizell and R. ~Golovko, {\em Legendrian submanifolds from Bohr-Sommerfeld covers of monotone Lagrangian tori}, 
Communications in Analysis and Geometry, Vol. 31, No. 4 (2023), 905--978.

\bibitem{DimitroglouRizellGolovko21}
G.~Dimitroglou Rizell and R.~Golovko, {\em
On Legendrian products and twist spuns},
Algebr. Geom. Topol. 21 (2021), no. 2, 665--695.

\bibitem{DingGeiges08}
F.~Ding and H.~Geiges, {\em Handle moves in contact surgery diagrams}, Journal of Topology, Volume 2, No. 1, 105--122 (2009).

\bibitem{DingGeiges10}
F.~Ding and H.~Geiges, {\em The diffeotopy group of $S^1 \times S^2$ via contact topology}, Compositio Math. {\bf 146}  1096--1112 (2010).


\bibitem{EkholmEtnyreSullivan05}
T.~Ekholm, J.~Etnyre and M.~Sullivan, {\em Non-isotopic Legendrian submanifolds in $\R^{2n+1}$},  J. Diff. Geom. {\bf 71} (2005), 85--128.

\bibitem{EkholmLekili}
T.~Ekholm and Y.~Lekili, {\em Duality between {L}agrangian and {L}egendrian invariants}, Geom. Topol. {\bf 27} no 6 (2023), 2049--2179.

\bibitem{WeinsteinRevisited}
Y.~Eliashberg, {\em Weinstein manifolds revisited.} Modern geometry: a celebration of the work of Simon Donaldson,  Proceedings of Symposia in Pure Mathematics, Vol. 99, 59--82, American Mathematical Society, Providence, RI, 2018.


\bibitem{EliashbergGanatraLzarev20}
Y.~Eliashberg, S.~Ganatra and O.~Lazarev,
{\em Flexible Lagrangians},  International Mathematics Research Notices, Volume 2020, Issue 8, April 2020, Pages 2408--2435.

\bibitem{EliashbergLazarevMurphy}
Y.~Eliashberg, O.~Lazarev and E.~Murphy, work in progress.


\bibitem{EliashbergMishachevCielebak}
Y.~Eliashberg, N.~Mishachev and K~Cielebak, {\em Introduction to the {$h$}-principle: Second Edition}, volume~48 of {\em Graduate Studies in Mathematics}.  Graduate Studies in Mathematics, Volume: 239; 2024; 363 pp. 

\bibitem{LagCaps}
Y.~Eliashberg and E.~Murphy, {\em Lagrangian caps},
Geom. Funct. Anal. 23(5):1483--1514 (2013).

\bibitem{EtnyreNg18}
J.~Etnyre and L.~Ng, {\em Legendrian contact homology in $\R^3$}, Surveys in Differential Geometry, 25, 103--161 (2020). 

\bibitem{GPS}
S.~Ganatra, J.~Pardon and V.~Shende, {\em Sectorial descent for wrapped Fukaya categories},
J. Amer. Math. Soc., DOI: https://doi.org/10.1090/jams/1035 

\bibitem{GanatraPardonShende}
S.~Ganatra, J.~Pardon and V.~Shende, {\em Covariantly Functorial Wrapped Floer Theory on Liouville Sectors},
Publ. Math. I.H.E.S. 131, 73–200 (2020).


\bibitem{GerstenhaberRothaus}
M.~Gerstenhaber and O.~Rothaus, {\em The Solution of Sets of Equations in Groups},
Proceedings of the National Academy of Sciences of the United States of America, Vol. 48, No. 9 (Sep. 15, 1962), pp. 1531--1533.

\bibitem{Golovko14}
R.~Golovko, {\em A note on the front spinning construction}, Bulletin of the London Mathematical Society, 46 (2014), no. 2, 258--268.

\bibitem{Gompf94}
R.~E.~Gompf, {\em Handlebody construction of Stein surfaces}, Ann. of Math. {\bf 48}, no. 2 (1994) 619--693.

\bibitem{Haefliger}
A. ~Haefliger, {\em Differentiable imbeddings}, Bull. Amer. Math. Soc. 67(1): 109--112, 1961. 

\bibitem{Husin}
A.~Husin, {\em Maslov class of exact Lagrangians and cylindrical handles}, \emph{arXiv preprint} arXiv:2502.14750, 2025. Available at: \url{https://arxiv.org/abs/2502.14750}.\color{black}

\bibitem{Kirby}
R.~C. Kirby, {\em The Topology of 4-Manifolds}, Lecture Notes in Math. 1374,
Springer-Verlag, Berlin (1990).

\bibitem{LambertCole14}
P.~Lambert-Cole, {\em Invariants of Legendrian products}, PhD thesis, Louisiana State University, 2014.

\bibitem{Lin}
F.~Lin, {\em Exact {L}agrangian caps of {L}egendrian knots}, J. Symplectic Geom. 14(1):269--295 (2016).

\bibitem{Murphy}
E.~Murphy, {\em Loose Legendrian embeddings in high dimensional contact manifolds}, preprint 2012, arXiv:1201.2245.


\bibitem{ODonnolPavelescu2012}
  D.~O'Donnol and E.~Pavelescu, \emph{On Legendrian graphs}, Algebraic \& Geometric Topology, \textbf{12} (2012), 1273--1299.
  \color{black}

\bibitem{RutherfordSulivan20}
D.~Rutherfod and M.~Sullivan,
{\em Cellular Legendrian contact homology for surfaces, part I},  Adv. Math. 374 (2020), 107348, 71 pp.

\bibitem{UpppwerboundLagrcob}
J. Sabloff, D. Shea Vela-Vick, C.-M. Michael Wong, 
{\em Upper bounds for the Lagrangian cobordism relation on Legendrian links}, preprint 2021. Available at arXiv:2105.02390.

\bibitem{SabloffVelaVickWongWu}
  J.~M.~Sabloff, D.~S.~Vela-Vick, C.-M.~M.~Wong, and A.~Wu, {\em Lagrangian Zigzag Cobordisms}, Journal of Symplectic Geometry, vol.~21, no.~1, pp.~1--52, 2023.
\color{black}

\bibitem{sumners}
D.~W.~Sumners, {\em Inversitble knot cobordisms}, Comment. Math. Helv., 46:240--256, 1971

\bibitem{Strakos}
  F.~Strako\v{s}, {\em K\"{u}nneth's formula for Legendrian spuns}, work in progress.

\bibitem{Weinstein}
A.~Weinstein, {\em Contact surgery and symplectic handlebodies}, Hokkaido Math. J., 20(2): 241--251, 1991.\color{black}







\end{thebibliography}
\end{document}